\newcommand{\calM}{\mathcal{M}}

\newcommand{\calO}{\mathcal{O}}

\newcommand{\bbC}{\mathbb{C}}

\newcommand{\bbP}{\mathbb{P}}
\newcommand{\bbQ}{\mathbb{Q}}

\newcommand{\bbZ}{\mathbb{Z}}

\newcommand{\Tr}{\textup{Tr}}
\newcommand{\rank}{\textup{rank}}
\newcommand{\Pic}{\textup{Pic}}

\def\Ker{{\text{Ker}}}

\newcommand{\DR}{\text{DR}}

\newcommand{\crys}{\text{crys}}

\newcommand{\NS}{\text{NS}}
\newcommand{\Br}{\text{Br}}

\newcommand{\GL}{\text{GL}}

\documentclass[11pt]{amsart}
\usepackage{amscd, amsmath, amsthm, amssymb}
\newtheorem{theorem}{Theorem}[section]
\newtheorem{lemma}[theorem]{Lemma}
\newtheorem{proposition}[theorem]{Proposition}

\newtheorem{corollary}[theorem]{Corollary}
\theoremstyle{definition}     % italic or bold etc.

\newtheorem{example}[theorem]{Example}

\newtheorem{remark}[theorem]{Remark}
\numberwithin{equation}{section}
\begin{document}

\title[Orders of automorphisms of K3 surfaces]{Orders of automorphisms of K3 surfaces}

\author[J. Keum]{JongHae Keum}
\address{School of Mathematics, Korea Institute for Advanced Study, Seoul 130-722, Korea } \email{jhkeum@kias.re.kr}
\thanks{Research supported by National Research Foundation of Korea (NRF grant).}


\date{July 30 2012, revised September 2013}
\begin{abstract} We determine all orders of automorphisms of complex K3 surfaces and of K3 surfaces in characteristic $p>3$.
%The set ${\rm Ord}_{\bbC}$ of all orders of automorphisms of finite order of complex K3 surfaces is given by $${\rm Ord}_{\bbC}=\{N\,|\, N\,{\rm is\,a\,positive\, integer},\,\phi(N)\le 20\},$$ where $\phi$ is the Euler function, and the set ${\rm Ord}_{p}$  of all orders of automorphisms of finite order of  K3 surfaces in characteristic $p$  by $${\rm Ord}_{p}=\left\{\begin{array}{lll} {\rm Ord}_{\bbC}&{\rm if}& p=7\,{\rm or}\,\, p>19,\\ {\rm Ord}_{\bbC}\setminus\{p,\, 2p\}&{\rm if}& p=13,\, 17,\, 19,\\ {\rm Ord}_{\bbC}\setminus\{44\}&{\rm if}& p=11,\\ {\rm Ord}_{\bbC}\setminus\{25,\, 50,\,60\}&{\rm if}& p=5.\\ \end{array} \right.$$
In particular,  66 is the maximum finite order in each
characteristic $p\neq 2,3$. As a consequence, we give a bound for the orders of finite groups acting on K3 surfaces in characteristic $p>7$.
\end{abstract}

\maketitle

\section{Introduction}

It is a  natural and fundamental problem to determine all possible
orders of automorphisms of K3 surfaces in any characteristic. Even
in the case of complex K3 surfaces, this problem has been settled
only for symplectic automorphisms and purely non-symplectic
automorphisms (Nikulin \cite{Nik}, Kond\={o} \cite{Ko}, Oguiso \cite{Og1}, Machida-Oguiso \cite{MO}). In
this paper we solve the problem in all characteristics except 2
and 3.
Let ${\rm Ord}_{\bbC}$ and ${\rm
Ord}_{p}$ be the sets of all orders of
automorphisms of finite order respectively of complex K3 surfaces and  K3 surfaces in characteristic $p$. Our main
result is the following:

\bigskip
\noindent {\bf Main Theorem.}
$${\rm Ord}_{\bbC}=\{N\,|\, N\,{\rm is\,a\,positive\,
integer},\,\phi(N)\le 20\},$$ \textit{where $\phi$ is the Euler
function, and}
$${\rm Ord}_{p}=\left\{\begin{array}{lll} {\rm Ord}_{\bbC}&{\rm if}& p=7\,{\rm or}\,\, p>19,\\
{\rm Ord}_{\bbC}\setminus\{p,\, 2p\}&{\rm if}& p=13,\, 17,\, 19,\\
{\rm Ord}_{\bbC}\setminus\{44\}&{\rm if}& p=11,\\
{\rm Ord}_{\bbC}\setminus\{25,\, 50,\,60\}&{\rm if}& p=5.\\
\end{array} \right.$$
\textit{In particular,  $66$ is the maximum finite order and ${\rm
Ord}_{p}\subset {\rm Ord}_{\bbC}$ in each characteristic $p\ge 5$.\\ In characteristic $p=2, 3$, a positive integer $N$ is the order of a  tame automorphism of a K3 surface if and only if $p\nmid N$ and $\phi(N)\le 20$.}

\medskip
The bound 66 is new even in characteristic 0. A previously known
bound in characteristic 0 is $528=8\times 66$ where 8 and 66 are
the bounds respectively for symplectic automorphisms of finite
order and purely non-symplectic automorphisms of finite order of
complex K3 surfaces (Nikulin \cite{Nik}). In characteristic $p>0$
bounds have been known only for special automorphisms, e.g.,
the bound 8 for tame symplectic automorphisms \cite{DK2}, the
bound 11 for $p$-cyclic wild automorphisms (Theorem 2.1, \cite{DK2}), and the
bound of Nygaard \cite{Ny}: for an Artin-supersingular K3
surface $X$ the image of ${\rm Aut}(X)$ in
 ${\rm GL}(H^0(X,
\Omega^2_X))$ is cyclic of order dividing $p^{\sigma(X)}+1$ where
$\sigma(X)$ is the Artin invariant of $X$ which may take values 1,
2,..., 10.

In any characteristic, an automorphism of a K3 surface $X$ is
called \textit{symplectic} if it acts trivially on the space
$H^0(X, \Omega^2_X)$ of regular 2-forms,  \textit{non-symplectic}
otherwise, and \textit{purely non-symplectic} if it acts
faithfully on the space $H^0(X, \Omega^2_X)$. A group of symplectic
automorphisms is called \textit{symplectic}.  In characteristic $p>0$, an automorphism
of finite order will be called \textit{tame} if its order is coprime to $p$ and \textit{wild} if
divisible by $p$.
%This notion is slightly different from \cite{DK2} and \cite{DK3} where $g$ is called wild if its order is a power of $p$ and a group $G$ is called wild if its order is divisible by $p$. Thus $g$ being wild in the new notion is equivalent to  the cyclic group $\langle g\rangle$ being wild in the old.

The list of finite groups which may act on a K3 surface is not yet
known. In positive characteristic the list seems much longer than
in characteristic 0. In fact, any finite group acting on a complex
K3 surface is of order $\le 3,840$ (Kond\=o \cite{Ko2}) and of
order $\le 960$ if symplectic (Mukai \cite{Muk}), while in
characteristic $p=5$ there is  a K3 surface with a wild action of the
simple group PSU$_3(5)$, whose order is $126,000$ (cf. \cite{DK2}), and
in characteristic $p=11$ a K3 surface with a wild action of the
Mathieu group $M_{22}$, whose order is $443,520$ (Dolgachev-Keum \cite{DK3}, Kond\=o \cite{Ko3}).
%, and in infinitely many positive characteristics a K3 surface with a tame symplectic action of the Mathieu group $M_{21}$ (Dolgachev-Keum \cite{DK2}).
%Note that $U_3(5)$ (resp. $M_{22}$) has an element of order 10 (resp. 11), but no tame symplectic automorphism with such an order exists.
Thus, it is natural to expect that in some characteristics $p$ the
set ${\rm Ord}_{p}$ would be bigger than the set ${\rm Ord}_{\bbC}$. Our
result shows that this never happens if $p>3$ (and not likely to happen even
in characteristic 2 or 3). It is interesting to note that elliptic
curves share such properties (Remark \ref{EC}).

The proof of Main Theorem will be divided into three cases: the
tame case (Theorem \ref{main-tame}), the complex case (Theorem
\ref{main-complex}) and  the wild case (Theorem \ref{main-wild}).
We first determine all orders of tame automorphisms of K3
surfaces.

\begin{theorem}\label{main-tame} Let $k$ be an algebraically closed field of
characteristic $p>0$.  Let $N$ be a positive integer not divisible
by $p$. Then $N$ is the order of an automorphism of a K3 surface
$X/k$ if and only if $\phi(N)\le 20$.
\end{theorem}

The if-part of Theorem \ref{main-tame} is proved in Section 3 by providing examples.  Indeed, over $k=\bbC$ motivated by the
Nikulin's fundamental work \cite{Nik}  Xiao \cite{Xiao}, Machida and
Oguiso \cite{MO}  proved that a positive integer $N$ is the order of a
purely non-symplectic automorphism of a complex K3 surface if and
only if $\phi(N)\le 20$ and $N\neq 60$. They also provided
examples (see \cite{MO} Proposition 4, also \cite{Ko}
Section 7 and \cite{Og1} Proposition 2.) In each of their examples
the K3 surface is defined over the integers and both the surface
and the automorphism have a good reduction mod $p$ as long as $p$
is coprime to the order $N$ and $p>2$ (Proposition \ref{exist}). Purely non-symplectic examples in characteristic $p=2$ are given in Example \ref{p2}. An example of a K3 surface with an automorphism of order 60
in any characteristic $p\neq 2, 3, 5$ (including the zero characteristic) is given in Example
\ref{5.12}. It can be shown (Keum \cite{K2}) that a K3 surface with a
tame automorphism of order 60 is unique up to isomorphism, and the
automorphism has non-symplectic order 12 and its 12th power is symplectic of order 5.

The only-if-part of Theorem \ref{main-tame} is proved in Section
4. Its proof depends on analyzing eigenvalues of the action on the
$l$-adic cohomology $H_{\rm et}^2(X,\bbQ_l)$, $l\neq p$. In
positive characteristic, the holomorphic Lefschetz formula is not
available and neither is the method using transcendental lattice.
Our proof extends to the complex case, yielding a proof free from
both, once we replace the $l$-adic cohomology by the singular
integral cohomology. In other words, the positive characteristic case and the zero characteristic case can be handled in a unified fashion.

\begin{theorem}\label{main-complex} A positive integer $N$ is the order of an automorphism of a complex K3 surface, projective or non-projective, if and only if
$\phi(N)\le 20.$
\end{theorem}

\begin{remark} By Deligne \cite{De} any K3 surface in
any positive characteristic lifts to characteristic 0. But
automorphisms of K3 surfaces do not in general. J.-P. Serre
\cite{Serre} has proved the following result about lifting to
characteristic 0: if $X$ is a smooth projective variety over an
algebraically closed field $k$ of characteristic $p$ with $H^2(X,
\calO_X)=H^2(X, \Theta_X)=0$, where $\Theta_X$ is the tangent
sheaf of $X$, and if $G\subset$ Aut$(X)$ is a finite tame
subgroup, then the pair $(X, G)$ can be lifted to the ring $W(k)$
of Witt vectors. K3 surfaces do not satisfy the condition and the
lifting theorem does not hold for them.  In fact, two groups in Dolgachev and Keum's classification (Theorem 5.2 \cite{DK2})  of finite groups which
may act tamely and symplectically on a K3 surface in positive
characteristic have orders $> 3,840$,
hence cannot act faithfully on any complex K3 surface. 
%In particular, there is a tame symplectic automorphism that does not lift to characteristic 0.
\end{remark}

In Section 6, we shall prove the faithfulness of the representation of the
automorphism group ${\rm Aut}(X)$ on the $l$-adic cohomology
$H_{\rm et}^2(X,\bbQ_l)$, $l\neq p$, which is a result of Ogus
(Corollary 2.5 \cite{Ogus}) when either $p\neq 2$ or $X$ is not superspecial.
%(the Hodge and conjugate filtrations on $H_{\DR}^2(X)$ coincide)

\begin{theorem}\label{faith} Let $X$ be a K3 surface over  an algebraically closed field $k$ of
characteristic $p> 0$. Let $l\neq p$ be a prime. Then the representation $${\rm Aut}(X)\to
{\rm GL}(H^2_{\rm et}(X,{\bbQ}_l)),\,\,g\mapsto g^*,$$ is faithful. 
\end{theorem}

In fact, Ogus proved under the condition the
faithfulness of the representation on the crystalline cohomology
$H_{\crys}^2(X/W)$, and it is known \cite{Ill} that the
characteristic polynomial of any automorphism has integer
coefficients which do not depend on the choice of cohomology.

An wild automorphism of a K3 surface exists only in
characteristic $p\le 11$ (Theorem 2.1 \cite{DK2}).

\begin{theorem} \label{main-wild} Let $X$ be a K3 surface over an algebraically closed field $k$ of
characteristic $p=11$, $7$ or $5$. Let $N$ be the order of an
automorphism of finite order of $X$. Assume that the order $N$ is
divisible by $p$. Then $$N = pn$$ where $n=1,\,2,\,3,\,6$ if
$p=11$, $n=1,\,2,\,3,\,4,\,6$ if $p=7$, $n=1,\,2,\,3,\,4,\,6,\, 8$ if $p=5$. All these orders
are supported by examples $($Examples \ref{exam11}, \ref{exam7},
\ref{exam5} and \ref{exam5-2}$)$.
\end{theorem}

The proof of Theorem \ref{main-wild} is given in theorems
\ref{11main}, \ref{7main} and \ref{5main}, and is based on the
faithfulness of Ogus (Theorem \ref{faith}) and the results on wild
$p$-cyclic actions on K3 surfaces  \cite{DK1}, \cite{DK2},
\cite{DK3}. A result of Deligne and Lusztig (Theorem 3.2 \cite{DL}), an extension of the Lefschetz fixed point formula to the case of wild automorphisms, is quite useful to rule out certain possibilities for orders, but not always. Many possible orders 
cannot be ruled out by Deligne and Lusztig, and require different arguments.

In any characteristic $p\ge 0$ there are automorphisms of infinite
order. For example, an elliptic K3 surface with Mordell-Weil rank
positive always admits automorphisms of infinite order, namely,
the automorphisms induced by the translation by a non-torsion
section of the Mordell-Weil group of the Jacobian fibration. Such
automorphisms are symplectic. Supersingular K3 surfaces (Ito
\cite{Ito}) and Kummer's quartic surfaces in characteristic $p\neq
2$ (Ueno \cite{Ueno86}) also admit symplectic automorphisms of
infinite order. Non-symplectic automorphisms of infinite order
also exist in characteristic $p\neq 2$, e.g., on a generic
Kummer's quartic surface the composition of an odd number of, more
than one, projections \cite{K}, where a projection is the
involution obtained by projecting the surface from a node onto
$\bbP^2$.

For a finite group $G$ acting on a K3 surface $X$, we denote by $G^*$ the image 
$$G^*\subset {\rm GL}(H^0(X, \Omega^2_X))\cong k^{\times}.$$
It is a cyclic group and its order $|G^*|$ is called a \textit{transcendental value}. 

Main Theorem together with the result of \cite{K2}, also with Proposition \ref{exist} and Example \ref{p2}, determines the set ${\rm TV}_{p}$ of all transcendental values in characteristic $p$.

\begin{corollary} In each characteristic $p>0$
$${\rm TV}_{p}=\{N\,|\, N\,{\rm is\,a\,positive\,
integer},\,p\nmid N,\,N\neq 60,\,\phi(N)\le 20 \}.$$
\end{corollary}

Note that 
$$\max{\rm TV}_{p}=\left\{\begin{array}{lll} 66&{\rm if}& p>3,\,\, p\neq 11,\\
54&{\rm if}& p=11,\\
50&{\rm if}& p=3,\\
33&{\rm if}& p=2.\\
\end{array} \right.$$
If a finite group $G$ acts on a supersingular K3 surface with Artin invariant 1, then $|G^*|$ divides $p+1$ (Nygaard \cite{Ny}), so belongs to the subset
$${\rm B}_{p}:=\{N\,|\, N\,{\rm is\,a\,positive\,
integer},\,N\mid (p+1),\,N\neq 60,\,\phi(N)\le 20 \}\subset {\rm TV}_{p}.$$ Define
$$\beta_p:=\max{\rm B}_{p}\le 66.$$
Note that $$\beta_p=p+1\,\,{\rm if}\,p\le 53,\,\, \beta_{59}=30,\,\, \beta_{61}=\beta_{73}=2,\,\, \beta_{67}=34,\,\, \beta_{71}=36,\, {\rm etc}.$$

In \cite{DK2} finite groups of tame symplectic automorphisms of K3 surfaces in
positive characteristic have been classified. The list consists of the groups appearing in the complex case (Mukai's list) and 27 new groups, called exceptional groups. A tame symplectic subgroup $G\subset {\rm Aut}(X)$ is exceptional iff the K3 surface $X$ is supersingular with Artin invariant 1 and the quotient $X/G$ is birational to $X$.
Among Mukai's list the Mathieu group $M_{20}$ has the largest order 960, and among exceptional groups $M_{21}$ has the largest order $20,160=21\times 960=2^6.3^2.5.7$. 

Recall that in char $p>11$ no K3 surface admits a wild automorphism. Main Theorem, with the results of \cite{DK2},  \cite{DK3}, \cite{Ny}, yields the following bounds:

\begin{theorem} Let $G$ be a finite group acting on a K3 surface $X$ in characteristic $p$. If $p>11$, then
$$|G|\le\left\{\begin{array}{ll} \beta_p |M_{21}|&{\rm if}\,\, G_{Sym}\,\,{\rm is\,\,exceptional},\\
66 |M_{20}|&{\rm otherwise} \\
\end{array} \right.$$ where $G_{Sym}$ is the maximal symplectic subgroup of $G$.
If $p=11$, then $$|G|\le\beta_{11} |M_{22}|=12\times 443,520.$$
\end{theorem}

 The above bound in characteristic $p=11$ is obtained from the classification in \cite{DK3} of all finite symplectic wild groups. The list consists of 5 groups: ${\bbZ}_{11}$, $11:5$, $L_2(11)$, $M_{11}$, $M_{22}$. Since $M_{22}$ contains $M_{21}$, an exceptional group, the K3 surface must be supersingular with Artin invariant 1. Finally, $54 |M_{20}|<12 |M_{22}|$. In characteristic $p\le 7$, symplectic wild groups have not been classified. 

\begin{remark}\label{beta}  The bounds do not seem to be optimal for general $p$. But for some $p$,  $\beta_p$ is very small, so the corresponding bound is close to optimal.
\end{remark}

\begin{remark}\label{EC}  (1) It is well known that any group automorphism
of an elliptic curve $E$ is of finite order and acts faithfully on
the first cohomology of $E$. The set of all possible orders of
group automorphisms of elliptic curves does not depend on the
characteristic and is given by
$$\{1,2,3,4,6\}=\{N\,|\, N\,{\rm is\,a\,positive\, integer},\,\phi(N)\le b_1=2\},$$ where $b_1$ is the first Betti number  of an elliptic curve.
But in characteristic 2 and 3, more groups may act on an elliptic
curve. In fact, the quotient group ${\rm Aut}(E)/{\rm Aut}_0(E)$
where ${\rm Aut}_0(E)\cong E$ is the identity component depends on
the $j$-invariant of $E$ and is isomorphic to the cyclic group
${\bbZ}_2$, ${\bbZ}_4$ or ${\bbZ}_6$ in characteristic $>3$, to
${\bbZ}_2$ or ${\bbZ}_3\rtimes {\bbZ}_4$ in characteristic $3$, to
${\bbZ}_2$ or $Q_8\rtimes {\bbZ}_3$ in characteristic $2$ (see
\cite{Silverman}).

(2)  Unlike the elliptic curve case there are abelian varieties
$X$ with an automorphism $g$ that acts faithfully on the first
cohomology of $X$, but $\phi({\rm ord}(g))> b_1(X)$. For example
the $n$-fold product $E^n$ of a general elliptic curve $E$ with
itself has ${\rm Aut}(E^n)/{\rm Aut}_0(E^n)=GL_n(\bbZ)$ where
${\rm Aut}_0(E^n)\cong E^n$ is the identity component. Take $n=10$
and $$g=(g_5, g_7)\in GL_4(\bbZ)\times GL_6(\bbZ)\subset
GL_{10}(\bbZ),$$ where
\begin{displaymath}
g_q = \left( \begin{array}{ccccc}
0 & 0 & \cdots & \cdots &-1 \\
1 & 0 & \cdots & \cdots& -1 \\
0& 1 & \cdots  & \cdots&-1  \\
\vdots & \vdots & \ddots & \vdots & \vdots\\
0& 0 & \cdots &1 & -1 \\
\end{array} \right)\in GL_{q-1}(\bbZ)
\end{displaymath}
is an element of order $q$. Then ${\rm ord}(g)=35$ and $\phi({\rm
ord}(g))> b_1(E^{10})=20$.

(3) K3 surfaces have the second Betti number $b_2(X)=22$. Note that
$$\phi(N)\le 22\Longleftrightarrow \phi(N)\le 20.$$
\end{remark}

\bigskip\noindent{\bf Question.}
K3 surfaces in characteristic  $\neq 2, 3$ and elliptic
curves in any characteristic satisfy the inequality $$\phi({\rm ord}(g^*|H^j(X)))\le b_j(X)$$ for
any automorphism $g$ of the variety $X$ whose induced action $g^*|H^j(X)$  on the $j$-th cohomology is of finite order. 
Some abelian varieties do not satisfy the property. Enriques surfaces probably do. Do Calabi-Yau 3-folds
so?

\bigskip\noindent{\bf Notation}

\medskip
\begin{itemize}
%\item $\omega_X$ : a nowhere vanishing regular 2-form on $X$;
%\medskip
\item ${\rm NS}(X)$ : the N\'eron-Severi group of a smooth projective variety $X$;
\medskip
\item $X^g={\rm Fix}(g)$ : the fixed locus of an automorphism $g$ of $X$;
\medskip
\item $e(g):=e({\rm Fix}(g))$, the Euler characteristic of ${\rm Fix}(g)$ for $g$ tame;
\medskip
\item $\Tr(g^*|H^*(X)):=\sum_{j=0}^{2\dim X} (-1)^j\Tr (g^*|H^j_{\rm et}(X,{\bbQ}_l))$.
 
\medskip\noindent
For an automorphism $g$ of a K3 surface $X$,
\end{itemize}
\begin{itemize}
\item ${\rm ord}(g)=m.n$ : $g$ is of order $mn$ and the natural homomorphism $\langle g\rangle\to {\rm GL}(H^0(X, \Omega^2_X))\cong k^*$ has kernel of order $m$ and image of order $n$;
\medskip
\item $[g^*]=[\lambda_1, \ldots, \lambda_{22}]$ : the list of the eigenvalues of $g^*|H^2_{\rm et}(X,{\bbQ}_l)$.
\medskip
\item $\zeta_a$ : a primitive $a$-th  root of unity in $\overline{\bbQ_l}$.
\medskip
\item $\phi$ : the Euler function.
\medskip
\item $[\zeta_a:\phi(a)]\subset [g^*]$ : all primitive $a$-th  roots of unity appear in $[g^*]$ where $\phi(a)$ indicates the number of
them.
\medskip
\item $[\lambda.r]\subset [g^*]$ : $\lambda$ repeats $r$ times in
$[g^*]$.
\medskip
\item $[(\zeta_a:\phi(a)).r]\subset [g^*]$ : the list $\zeta_a:\phi(a)$ repeats $r$ times in
$[g^*]$.
\medskip
\item For $b$ odd,  $\pm \zeta_{b}:\phi(b)=\zeta_{b}:\phi(b)$ or $\zeta_{2b}:\phi(b)$.
\end{itemize}

\section{Preliminaries}

\begin{proposition}\label{integral} Let $X$ be a projective variety over an algebraically closed field $k$ of
characteristic $p> 0$. Let $g$ be an automorphism of $X$. Let $l\neq p$. Then
the following hold true.
\begin{enumerate}
\item $($3.7.3 \cite{Ill}$)$ The characteristic polynomial of
$g^*|H_{\rm et}^j(X,\bbQ_l)$ has integer coefficients for each
$j$. Furthermore  the characteristic polynomial  does not depend on the choice of  cohomology, $l$-adic or crystalline.
\item  If $g$ is of finite order, then $g$ has an
invariant ample divisor, hence  $g^*|H_{\rm et}^2(X,\bbQ_l)$ has
$1$ as an eigenvalue.
 \item  If $X$ is a K3 surface and $g^*$
acts trivially on $H_{\rm et}^2(X,\bbQ_l)$, then $g^*$ acts
trivially on the space of regular 2-forms  $H^0(X,\Omega_X^2)$.
\item  If $X$ is a K3 surface, $g$ is tame and $g^*|H^0(X,\Omega_X^2)$ has
$\zeta_n\in k$ as an eigenvalue, then $g^*|H_{\rm et}^2(X,\bbQ_l)$ has $\zeta_n\in \overline{\bbQ_l}$ as an
eigenvalue.
\end{enumerate}
\end{proposition}

\begin{proof}
(2) For any ample divisor $D$ the sum $\sum g^{i}(D)$ is $g$-invariant.  A $g^*$-invariant ample line bundle gives a $g^*$-invariant vector in the 2nd crystalline cohomology $H_{\rm crys}^2(X/W)$ under the Chern class map 
$$c_1: \Pic(X)\to H_{\rm crys}^2(X/W).$$ Thus $1$ is an eigenvalue of
$g^*|H_{\rm crys}^2(X/W)$. Now apply (1). Here $W=W(k)$ is the ring of Witt vectors. See \cite{Ill} for the  crystalline cohomology.

(3) Consider the representation of $\langle g\rangle$ on the
second crystalline cohomology $H_{\crys}^2(X/W)$. The quotient module
$$H_{\crys}^2(X/W)/pH_{\crys}^2(X/W)$$ is a finite dimensional
$k$-vector space isomorphic to the algebraic de Rham cohomology
$H_{\DR}^2(X)$. Thus  the representation of $\langle g\rangle$ on
 $H_{\DR}^2(X)$ is also trivial. It is known that the Hodge to de Rham spectral sequence  $$E_1^{t,s}:=H^s(X, \Omega_X^t)\Rightarrow H_{\DR}^*(X)$$ degenerates at $E_1$, giving
 the Hodge filtration on $H_{\DR}^*(X)$  and the following canonical exact sequences:
$$0\to F^1\to F^0=H_{\DR}^2(X)\to H^2(X, \mathcal{O}_X)\to 0$$
 $$0\to F^2=H^0(X,\Omega_X^2)\to F^1\to  H^1(X,\Omega_X^1)\to 0.$$ In particular $g^*$
acts trivially on the space of regular 2-forms  $H^0(X,\Omega_X^2)$.

(4) Since $g^*|H^0(X,\Omega_X^2)$ has
$\zeta_n\in k$ as an eigenvalue, so does $g^*|H_{\DR}^2(X)$. The corresponding eigenvalue of  $g^*|H_{\crys}^2(X/W)$ must be an $np^r$-th root of unity for some $r$, since $n$ is not divisible by $p$. Then $g^{p^r*}|H_{\crys}^2(X/W)$ has an  $n$-th root of unity as an eigenvalue.
Since $g$ is tame, so does $g^{*}|H_{\crys}^2(X/W)$.
\end{proof}

Recall that for a nonsingular projective variety $Z$ in
characteristic $p>0$, there is an exact sequence of
$\bbQ_l$-vector spaces
\begin{equation}\label{trans}
0\to \NS(Z)\otimes \bbQ_l \to  H_{\rm et}^2(Z,\bbQ_l) \to
T_l^2(Z)\to 0
\end{equation}
where $T_l^2(Z) = T_l(\textup{Br}(Z))$ in the standard notation in
the theory of \'etale cohomology. The Brauer group $\Br(Z)$ is
known to be a birational invariant, and it is trivial
when $Z$ is a rational variety.  In fact, one can show that
$$\NS(Z)\otimes \bbQ_l = \Ker (H_{\rm et}^2(Z,\bbQ_l)\to H^2(k(Z),\bbQ_l));$$
$$T_l^2(Z) = \textup{Im}(H_{\rm et}^2(Z,\bbQ_l)\to H^2(k(Z),\bbQ_l)).$$
Here $H^2(k(Z),\bbQ_l) = \underrightarrow{\lim}_U H^2(U,\bbQ_l),$
where $U$ runs through the set of open subsets of $Z$ (see
\cite{Shioda2}).  It is known that the dimension of all
$\bbQ_l$-spaces from above do not depend on $l$ prime to the
characteristic $p$.

\begin{proposition}\label{diminv} In the situation as above, let $g$ be an automorphism of $Z$
of finite order. Assume $l\ne p$. Then the following assertions
are true.
\begin{enumerate}
\item  Both traces of $g^*$ on ${\rm NS}(Z)$ and on $T_l^2(Z)$ are integers.
\item  $\rank~{\rm NS}(Z)^g=\rank~{\rm NS}(Z/\langle  g \rangle).$
\item  $\dim H^2_{\rm et}(Z,{\bbQ}_l)^g=\rank~{\rm NS}(Z)^g+\dim T_l^2(Z)^g$.
\item  If  the minimal resolution $Y$ of $Z/\langle  g \rangle$ has $T_l^2(Y)=0$, then $$\dim H^2_{\rm et}(Z,{\bbQ}_l)^g=\rank~{\rm NS}(Z)^g.$$
\end{enumerate}
The condition of $(4)$ is satisfied if $Z/\langle g\rangle$ is
rational or is birational to an Enriques surface.
\end{proposition}

\begin{proof} (4) By Proposition 5 \cite{Shioda2}, $T_l^2(Z)^g\cong
T_l^2(Y)=0.$
Hence the result follows from (3).
\end{proof}

We use the following result of Deligne and Lusztig which is an extension to the case of wild automorphisms of the Lefschetz fixed point formula.

\begin{proposition}\label{DL}$($Theorem 3.2 \cite{DL}$)$ Let $X$ be a scheme,
 separated and of finite type over  an algebraically closed field $k$ of
characteristic $p> 0$ and let $g$ be an automorphism of finite
order of $X$. We decompose $g$ as $g=su$ where $s$ and $u$ are
powers of $g$ respectively of order prime to $p$ and
 a power of $p$. Then $$\Tr(g^*|
H_c^*(X)) =\Tr(u^*| H_c^*(X^s))$$
where the cohomology is the $l$-adic rational cohomology with
compact support.
\end{proposition}

\begin{proposition}\label{trace}$($Lefschetz fixed point formula$)$   Let $X$ be a smooth projective variety over  an algebraically closed field $k$ of
characteristic $p> 0$ and let $g$ be a tame automorphism of $X$. Then $X^g={\rm Fix}(g)$ is smooth and
 $$e(g):=e({\rm Fix}(g))=\Tr(g^*|
H_{\rm et}^*(X)).$$
\end{proposition}

\begin{proof} This follows from Proposition \ref{DL}. The tame case is just the case with $s=g$ and $u=1_X$. Note that
$\sum_j (-1)^j\Tr(1^*| H_{\rm et}^j(X^g, \bbQ_l))=e(X^g).$
\end{proof}

\begin{proposition}\label{sym}$($Theorem 3.3  and Proposition 4.1 \cite{DK2}$)$ A tame
symplectic automorphism $h$ of  a K3 surface has finitely many
fixed points, the number of fixed points $f(h)$ depends only on
the order of $h$ and the list of possible pairs $({\rm ord}(h),
f(h))$ is the same as in the complex case:
$$ ({\rm ord}(h), f(h))=(2,8),\,\,(3,6),\,\,(4,4),\,\,(5,4),\,\,(6,2),\,\,(7,3),\,\,(8,2).$$
\end{proposition}

\begin{lemma}\label{Lefschetz} Let $h$ be a tame symplectic automorphism of  a K3 surface $X$. Then $h^*|H_{\rm et}^2(X,\bbQ_l)$ has eigenvalues
$$\begin{array}{lll} {\rm ord}(h)=2&:&[h^*]=[1,\, 1.13,\, -1.8]\\
{\rm ord}(h)=3&:&[h^*]=[1,\, 1.9,\, (\zeta_3:2).6]\\
{\rm ord}(h)=4&:&[h^*]=[1,\, 1.7,\, (\zeta_4:2).4,\, -1.6]\\
{\rm ord}(h)=5&:&[h^*]=[1,\, 1.5,\, (\zeta_5:4).4]\\
{\rm ord}(h)=6&:&[h^*]=[1,\, 1.5,\, (\zeta_3:2).4,\, (\zeta_6:2).2,\, -1.4]\\
{\rm ord}(h)=7&:&[h^*]=[1,\, 1.3,\, (\zeta_7:6).3]\\
{\rm ord}(h)=8&:&[h^*]=[1,\,
1.3,\,(\zeta_8:4).2,\,(\zeta_4:2).3,\, -1.4]\end{array}$$ where
the first eigenvalue corresponds to an invariant ample divisor.
\end{lemma}

\begin{proof} The result follows from  Proposition \ref{trace} and Proposition \ref{sym}.
\end{proof}

\begin{lemma}\label{nsym4}  Let $g$ be a tame automorphism of finite order of a K3 surface $X$ such that $g^n$ is symplectic. Then for any $i$ dividing $n$
${\rm Fix}(g^i)$ is finite and $0\le e(g^i)=\Tr(g^{i*}|H^*(X))\le e(g^n)=\Tr(g^{n*}|H^*(X))$.
\end{lemma}

\begin{proof}  ${\rm Fix}(g^i)\subset {\rm Fix}(g^n)$ and the latter is a finite set.
\end{proof}

The following lemma will be used in the proof that no automorphism of order 44 appears in characteristic $p=11$ (Lemma \ref{44}).

\begin{lemma}\label{nsym2}  
Let $X$ be a K3 surface in characteristic $p\neq 2$.
Assume that $h$ is an automorphism of order $2$ with $\dim
H^2_{\rm et}(X,{\bbQ}_l)^h =2$. Then $h$ is non-symplectic and has
an $h$-invariant elliptic fibration $\psi:X\to {\bf P}^1$,  the quotient $X/\langle h\rangle\cong {\bf F}_e$ a rational ruled surface, and $X^h$ is either a curve of genus
$9$ which is a $4$-section of $\psi$ or the union of a section and a curve of genus $10$ which
is a $3$-section.
In the first case $e\le 2$, and in the second $e=4$. 
%Each singular fibre of $\psi$ is of type $I_1$ $($nodal$)$, $I_2$, $II$ $($cuspidal$)$ or $III$, and is intersected by $X^h$ at the node and two smooth points if  of type $I_1$, at the two singular points if of type $I_2$, at the cusp with multiplicity $3$ and a smooth point if of type $II$, at the singular point tangentially to both components if of type $III$. If $X^h$ contains a section, then each singular fibre is of type $I_1$ or $II$.
\end{lemma}

\begin{proof} Since $\dim H^2_{\rm et}(X,{\bbQ}_l)^h=2$, the eigenvalues of
$h^*|H^2_{\rm et}(X,{\bbQ}_l)$ must be $$[h^*]=[1.2,\,-1.20], \,\,{\rm so}\,\,\,\Tr
(h^*|H^*(X))=-16.$$ By Lemma \ref{Lefschetz},
$h$ is non-symplectic, thus $X^h$  is a disjoint union of smooth
curves and $X/\langle h\rangle$ is a nonsingular
rational surface.  By Proposition
\ref{diminv}, $X/\langle h\rangle$ has Picard number 2, hence is
isomorphic to a rational ruled surface ${\bf F}_e$.  Note that $e(X^h)=-16$, hence $X^h$ is non-empty
and has at most 2 components. Thus $X^h$ is either a curve $C_9$ of
genus 9 or the union of two curves $C_0$ and $C_{10}$ of genus 0 and 10, respectively. 
 In the first case, the image $C_9'\subset {\bf F}_e$ of $C_9$ satisfies $C_9'^2=32$ and $C_9'K=-16$, hence $C_9'\equiv 4S_0+(4+2e)F$, where $S_0$ is the
section with $S_0^2=-e$, and $F$ a fibre of ${\bf F}_e$. Since $S_0C_9'\ge 0$, $e\le 2$.
In the second case, the image $C_0'$ of $C_0$ has $C_0'^2=-4$, hence $C_0'=S_0$ and $e=4$, then the image  $C_{10}'\equiv 3(S_0+4F)$. 

In characteristic $p\neq 3$ the pull-back of
the ruling on ${\bf F}_e$ gives an $h$-invariant elliptic fibration
$\psi:X\to {\bf P}^1$. 
%Each singular fibre has at most 2 components since it is the pull-back of a fibre of ${\bf F}_e$.

In characteristic $p=3$ we have to show that the pull-back is not a quasi-elliptic fibration. Suppose it is. The closure of the cusps of irreducible fibres is a smooth rational curve and must be fixed pointwise by $h$, then the genus 10 curve must be a section of the quasi-elliptic fibration, impossible. 
\end{proof}

We will use frequently the  Weyl theorem of the following form.

\begin{lemma}\label{Weyl} Let $V$ be a finite dimensional vector space over a field  of
characteristic $0$. Let $g\in \GL(V)$ be a linear automorphism of finite order.
Assume that the characteristic polynomial of $g$ has integer coefficients.
If for some positive integer $m$ a primitive $m$-th root of unity appears
with multiplicity $r$ as an eigenvalue of $g$, then so does
each of its conjugates.
\end{lemma}

\begin{proof}
Let $n={\rm ord}(g)$. Every eigenvalue of $g$ is an $n$-th root of unity, not necessarily primitive.
If  a primitive $m$-th root of unity appears as an eigenvalue of $g$, then its cyclotomic polynomial must divide
the characteristic polynomial of $g$.
\end{proof}

The following easy lemmas also will be used frequently.

\begin{lemma}\label{fix} Let $S$ be a set and ${\rm Aut}(S)$ be the group of bijections of $S$. For any $g\in {\rm Aut}(S)$ and positive integers $a$ and $b$,
\begin{enumerate}
\item ${\rm Fix}(g)\subset {\rm Fix}(g^a)$;
\item ${\rm Fix}(g^a)\cap {\rm Fix}(g^b)={\rm Fix}(g^d)$ where $d=\gcd (a, b)$;
\item ${\rm Fix}(g)= {\rm Fix}(g^a)$ if ${\rm ord}(g)$ is finite and prime to $a$.
\end{enumerate}
\end{lemma}

\begin{lemma}\label{P1} Let $g$ be an automorphism of
${\bf P}^1$ over an algebraically closed field $k$ of characteristic
$p>0$. \begin{enumerate}
\item If $g^{pm}=1$, $g^{m}\neq 1$ and $m$ is coprime to
$p$, then $g^{p}=1$. \item If $g^{p^2}=1$, then $g^{p}=1$.
\end{enumerate}
\end{lemma}

\begin{proof} easily follows from the Jordan
canonical form.
\end{proof}

\begin{lemma}\label{sum} Let $R(n)$ be the sum of all primitive $n$-th root of unity in $\overline{\bbQ}$ or in
$\overline{\bbQ_l}$ where $(l,n)=1$. Then
$$R(n)=\left\{\begin{array}{ccl} 0&{\rm if}& n\,{\rm has\,\, a\,\, square\,\, factor},\\
(-1)^t&{\rm if}& n\,{\rm is\,\, a\,\, product\,\, of}\,\,t\,\,{\rm distinct\,\, primes}.\\
\end{array} \right.$$
\end{lemma}

\begin{proof} Let $n=q_1^{r_1}q_2^{r_2}\cdots q_t^{r_t}$ be the prime factorization. It is easy to see
that a primitive $n$-th root of unity can be factorized uniquely
into a product of a primitive $q_1^{r_1}$-th root, ... , a
primitive $q_t^{r_t}$-th root of unity, thus
$$R(n)=R(q_1^{r_1})R(q_2^{r_2})\cdots R(q_t^{r_t}).$$ For any prime $q$, $R(q)=-1$ and $R(q^r)=0$ if $r>1$.
\end{proof}

\section{Examples of Automorphisms}

In this section we will prove the if-part of Theorem
\ref{main-tame} and Theorem \ref{main-complex} by providing
examples. See Propositions \ref{C-exist} and \ref{exist}.

Let $X$ be a complex K3 surface. The transcendental lattice $T_X$
of $X$ is by definition the orthogonal complement of the
N\'eron-Severi group in the cohomology lattice $H^2(X,\bbZ)$. If
$X$ is not projective, then the image of the homomorphism $${\rm
Aut}(X)\to\GL(H^0(X, \Omega^2_X))\cong \bbC^*$$ is either trivial or an
infinite cyclic group (cf. Ueno \cite{Ueno}). Thus every
automorphism of finite order of a non-projective K3 surface is
symplectic, hence of order $\le 8$ (Nikulin \cite{Nik}).

Let $X$ be a projective complex K3 surface and let $g$ be an
automorphism of non-symplectic order $n$, i.e., its image in
$\GL(H^0(X, \Omega^2_X))$ is of order $n$. We may regard the
transcendental lattice $T_X$ as a $\bbZ[\langle g\rangle]$-module
via the natural action of $\langle g\rangle$ on $T_X$. Since $g$
has non-symplectic order $n$, $T_X$ becomes a free $\bbZ[\langle
g\rangle]/\langle\Phi_n(g)\rangle$-module where $\Phi_n(x)\in
\bbZ[x]$ is the $n$-th cyclotomic polynomial \cite{Nik}, thus
$T_X$ can be viewed as a free $\bbZ[\zeta_n]$-module via the
isomorphism $$\bbZ[\langle g\rangle]/\langle\Phi_n(g)\rangle\cong
\bbZ[\zeta_n]$$ where $\zeta_n$ is a primitive $n$-th root of
unity. In particular $\phi(n)$ divides $\rank~T_X$ where $\phi$ is
the Euler function. Since $X$ is projective, $\rank~T_X\le 21$ and
hence
$$\phi(n)\le 21.$$ Motivated by this, Kond\=o \cite{Ko}, Xiao
\cite{Xiao}, Machida and Oguiso \cite{MO} studied purely
non-symplectic automorphisms of complex K3 surfaces and proved
that a positive integer $n$ is the order of a purely
non-symplectic automorphism of a complex K3 surface if and only if
$\phi(n)\le 20$ and $n\neq 60$. Their examples are displayed in Example \ref{pnonsym}.
For convenience, we list all integers $n$ with $\phi(n)\le 21$
in Table \ref{21}. There is no integer $n$ with $\phi(n)=21.$

\bigskip
\begin{table}[ht]
\caption{The list of all integers $n$ with $\phi(n)\le
21$}\label{21}
\renewcommand\arraystretch{1.2}
%\noindent\[
$$
\begin{array}{|c|l|l|l|l|l|c|c|c|l|l|}
\hline
 \phi(n) & 20&18&16&12&10&8&6&4&2&1\\ \hline
&66&54&60&42&22&30&18&12&6&2\\
&50&38&48&36&11&24&14&10&4&1\\
n&44&27&40&28&&20&9&8&3&\\
&33&19&34&26&&16&7&5&&\\
&25&&32&21&&15&&&&\\
&&&17&13&&&&&&
\\ \hline
\end{array}
$$
%\]
\end{table}

\bigskip
From Table
\ref{21} we see that a positive integer $n$ satisfying  $n\neq 60$ and
$\phi(n)\le 20$ is a divisor of an element of the set
$$\calM_{pns}:=\{66,\, 50,\, 44,\, 54,\, 38,\, 48,\, 40,\, 34,\, 32,\, 42,\, 36,\, 28,\, 26,\, 30\}.$$

\begin{example}\label{pnonsym}$($Proposition 4 \cite{MO}, also Section 7 \cite{Ko},  Proposition 2 \cite{Og1}$)$
For each $n\in \calM_{pns}$, there is a K3 surface $X_n$ with a
purely non-symplectic automorphism $g_n$ of order $n$. The surface
$X_n$ is defined by the indicated Weierstrass equation for $n\neq
50$, $40$. The surface $X_{50}\subset \bbP(1,1,1,3)$ is defined as
a double plane branched along a smooth sextic curve and the
surface $X_{40}$ as (the minimal resolution
of) a double plane branched along the union of a line and a smooth
quintic curve.

\bigskip
\begin{enumerate}
\item $X_{66}: y^2=x^3+t(t^{11}-1),\qquad g_{66}(t,x,y)=(\zeta_{66}^{54}t, \zeta_{66}^{40}x, \zeta_{66}^{27}y)$;

\medskip
\item $X_{50}: w^2=x^6+xy^5+yz^5,\quad g_{50}(x,y, z, w)=(x, \zeta_{50}^{40}y, \zeta_{50}^2z, \zeta_{50}^{25}w)$;

\medskip
\item $X_{44}: y^2=x^3+x+t^{11},\qquad g_{44}(t,x,y)=(\zeta_{44}^{34}t, \zeta_{44}^{22}x, \zeta_{44}^{11}y)$;

\medskip
\item $X_{54}: y^2=x^3+t(t^{9}-1),\qquad g_{54}(t,x,y)=(\zeta_{54}^{12}t, \zeta_{54}^{4}x, \zeta_{54}^{33}y)$;

\medskip
\item $X_{38}: y^2=x^3+t^7x+t,\qquad g_{38}(t,x,y)=(\zeta_{38}^{4}t, \zeta_{38}^{14}x, \zeta_{38}^{21}y)$;

\medskip
\item $X_{48}: y^2=x^3+t(t^{8}-1),\qquad g_{48}(t,x,y)=(\zeta_{48}^{6}t, \zeta_{48}^{2}x, \zeta_{48}^{3}y)$;

\medskip
\item $X_{40}: w^2=x(x^4z+y^5-z^5),\quad g_{40}(x,y, z, w)=(x, \zeta_{40}^2y, \zeta_{40}^{10}z,
\zeta_{40}^{5}w);$

\medskip
\item $X_{34}: y^2=x^3+t^7x+t^2,\qquad g_{34}(t,x,y)=(\zeta_{34}^{4}t, \zeta_{34}^{14}x, \zeta_{34}^{21}y)$;

\medskip
\item $X_{32}: y^2=x^3+t^2x+t^{11},\qquad g_{32}(t,x,y)=(\zeta_{32}^{2}t, \zeta_{32}^{18}x, \zeta_{32}^{27}y)$;

\medskip
\item $X_{42}: y^2=x^3+t^5(t^{7}-1),\qquad g_{42}(t,x,y)=(\zeta_{42}^{18}t, \zeta_{42}^{2}x, \zeta_{42}^{3}y)$;

\medskip
\item $X_{36}: y^2=x^3+t^5(t^{6}-1),\qquad g_{36}(t,x,y)=(\zeta_{36}^{30}t, \zeta_{36}^{2}x, \zeta_{36}^{3}y)$;

\medskip
\item $X_{28}: y^2=x^3+x+t^{7},\qquad g_{28}(t,x,y)=(\zeta_{28}^{2}t, \zeta_{28}^{14}x, \zeta_{28}^{7}y)$;

\medskip
\item $X_{26}: y^2=x^3+t^5x+t,\qquad g_{26}(t,x,y)=(\zeta_{26}^{4}t, \zeta_{26}^{10}x, \zeta_{26}^{15}y)$;

\medskip
\item $X_{30}: y^2=x^3+(t^{10}-1),\qquad g_{30}(t,x,y)=(\zeta_{30}^{3}t, \zeta_{30}^{10}x,
y)$.
\end{enumerate}
\end{example}

\bigskip
We give an example of a K3 surface with an order 60 automorphism.

\begin{example}\label{5.12}
In char $p\neq 2,3,5$, there is a K3 surface with an automorphism
of order $60=5.12$
$$X_{60}: y^2+x^3+t^{11}-t = 0,\qquad g_{60}(t,x,y)=(\zeta_{60}^6t, \zeta_{60}^2x, \zeta_{60}^3y).$$
The surface $X$ has 12 type
 $II$-fibres at $t=\infty$,  $t^{11}-t=0$.
\end{example}

\begin{remark}
In \cite{K2} it has been shown that in characteristic $p\ge 0$, $p\neq 2, 3, 5$, for a K3 surface $X$ with an
automorphism $g$ of order 60 the pair $(X, \langle g\rangle)$ is isomorphic to the pair $(X_{60}, \langle g_{60}\rangle)$.
\end{remark}

A positive integer $n$ with $\phi(n)\le 20$ is a divisor of an element of the set
$$\calM:=(\calM_{pns}\setminus\{30\})\cup\{60\}.$$

\begin{proposition}\label{C-exist}  For any positive integer $N$ satisfying
$\phi(N)\le 20$, there is a complex K3 surface with an
automorphism of order $N$.
\end{proposition}

\begin{proposition}\label{exist}  For any positive integer $N$ satisfying
$\phi(N)\le 20$ and any prime $p$ coprime to $N$, there is a K3
surface in characteristic $p$ with an automorphism of order $N$. Furthermore there is a purely non-symplectic example of order $N$, if $N\neq 60$.
\end{proposition}

\begin{proof} In each of the above examples the K3 surface $X_n$ is defined over the
integers. Fix a prime $p$. Then a positive integer $n$ with
$\phi(n)\le 20$, not divisible by $p$, is a divisor of an element
of the set ${\calM}_{p}$, where
$${\calM}_{p}:=\left\{\begin{array}{lll} {\calM}&{\rm if}& p>19,\\
 {\calM}\setminus\{2p\}&{\rm if}& p=13,\, 17,\, 19,\\
  {\calM}\setminus\{66,\, 44\}&{\rm if}& p=11,\\
   {\calM}\setminus\{42,\, 28\}&{\rm if}& p=7,\\
    {\calM}\setminus\{50,\,40,\, 60\}&{\rm if}& p=5,\\
    \{50,\,44,\, 38,\, 40,\, 34,\, 32,\, 28,\, 26\}&{\rm if}& p=3,\\
    \{33,\, 25,\, 27,\, 19,\, 17,\, 21,\, 13,\, 15\}&{\rm if}& p=2.
\end{array} \right.$$
When $p>2$, it is easy to check that for each $n\in \calM_{p}$ the pair $(X_n,
g_n)$ has a good reduction mod $p$, i.e., the surface $X_n$ defines
a K3 surface in characteristic $p$ and $g_n$ defines an
automorphism of order $n$. When $p=2$, the surface $X_{2n}$, $n\in \calM_{2}$, may have a non-rational double point, so does not define a K3 surface. In that case we provide a new example as below.
\end{proof}

\begin{example}\label{p2}
Assume that $p=2$. 
For each $n\in \calM_{2}$, there is a K3 surface $Y_n$ with a
purely non-symplectic automorphism $g_n$ of order $n$. The surface $Y_{25}\subset \bbP(1,1,1,3)$ is a weighted hypersurface, and $Y_{19}$ and $Y_{13}$ are given by  quasi-elliptic fibrations.

\bigskip
\begin{enumerate}
\item $Y_{33}: y^2+t^6y=x^3+t(t^{11}-1),\qquad g_{33}(t,x,y)=(\zeta_{33}^{21}t, \zeta_{33}^{7}x, \zeta_{33}^{27}y)$;

\medskip
\item $Y_{25}: w^2+x^3w=x^6+xy^5+yz^5,\quad g_{25}(x,y, z, w)=(x, \zeta_{25}^{15}y, \zeta_{25}^2z, w)$;

\medskip
\item $Y_{27}: y^2+t^5y=x^3+t(t^{9}-1),\qquad g_{27}(t,x,y)=(\zeta_{27}^{12}t, \zeta_{27}^{4}x, \zeta_{27}^{6}y)$;

\medskip
\item $Y_{19}=X_{38}: y^2=x^3+t^7x+t,\qquad g_{19}(t,x,y)=(\zeta_{19}^{4}t, \zeta_{19}^{14}x, \zeta_{19}^{2}y)$;

\medskip
\item $Y_{17}: y^2+ty=x^3+t^7x+t^2,\qquad g_{17}(t,x,y)=(\zeta_{17}^{4}t, \zeta_{17}^{14}x, \zeta_{17}^{4}y)$;

\medskip
\item $Y_{21}: y^2+t^6y=x^3+t^5(t^{7}-1),\qquad g_{21}(t,x,y)=(\zeta_{21}^{18}t, \zeta_{21}^{2}x, \zeta_{21}^{3}y)$;

\medskip
\item $Y_{13}=X_{26}: y^2=x^3+t^5x+t,\qquad g_{13}(t,x,y)=(\zeta_{13}^{4}t, \zeta_{13}^{10}x, \zeta_{13}^{2}y)$;

\medskip
\item $Y_{15}: y^2+t^5y=x^3+(t^{10}-1),\qquad g_{15}(t,x,y)=(\zeta_{15}^{3}t, \zeta_{15}^{10}x, y)$.
\end{enumerate}
\end{example}

\section{The Tame Case}

In this section we will prove the following:

\begin{proposition}\label{tame} Let $k$ be an algebraically closed field of
characteristic $p>0$.  Let $N$ be a positive integer not divisible
by $p$. If $N$ is the order of an automorphism of a K3 surface
$X/k$, then $\phi(N)\le 20$ where $\phi$ is the Euler
function.
\end{proposition}

Let $g$ be a tame automorphism of order $N$ of a K3 surface, i.e., the order $N$ is prime to the characteristic $p$.
Assume that
$${\rm ord}(g)=N=m.n$$ i.e., $g$ is of order $N=mn$ and its action
 on the space $H^0(X, \Omega^2_X)$ has order $n$. Then $g^n$ is symplectic of order $m$, hence by Proposition
\ref{sym},
$$ m\le 8.$$
By Proposition \ref{integral}(4),  $g^*|H_{\rm et}^2(X,\bbQ_l)$ has $\zeta_n\in \overline{\bbQ_l}$
as an eigenvalue. The second
cohomology space has dimension 22 and  $g^*$ has 1 as an
eigenvalue corresponding to an invariant ample divisor, so by Proposition \ref{integral} and Lemma \ref{Weyl},
$$\phi(n)\le 21.$$  See Table \ref{21} for all integers $n$ with $\phi(n)\le 21$.

We will prove $\phi(N)\le 20$ in the following lemmas
\ref{pn>13}---\ref{pn=2}.

\begin{lemma}\label{pn>13} If $\phi(n)> 13$, then $m=1$.
\end{lemma}

\begin{proof}
The primitive $n$-th root $\zeta_n$ is an eigenvalue of $g^*|H_{\rm et}^2(X,\bbQ_l)$, hence
$$[1,\,\zeta_n:\phi(n)]\subset[g^*]$$
where the first eigenvalue  $1\in [g^*]$ corresponds to an
invariant ample divisor (Proposition \ref{integral}),
$\zeta_n:\phi(n)$ means all primitive $n$-th roots of unity and
$\phi(n)$ indicates the number of them. Thus we infer that $g^{n}$
is symplectic of order $m$ such that
$$[1,\,1.\phi(n)]\subset[g^{n*}].$$
If $\phi(n)> 13$, then  $m=1$ by Lemma \ref{Lefschetz}.
\end{proof}

\begin{lemma}\label{pn=12} Assume that $\phi(n)=12$. Then $m=1$ if $n=42,\,\,36,\,\,28,\,\,26$, and $m\le 2$ if $n=21,\,\,13$.
\end{lemma}

\begin{proof} Since $[1,\,\zeta_n:\phi(n)]\subset[g^{*}]$, $[1,\,1.12]\subset[g^{n*}].$ Since $g^{n}$ is symplectic of
order $m$, $m\le 2$ by Lemma \ref{Lefschetz}.

 Assume
$m=2$. Then $g^{n}$ is symplectic of order 2 and by Lemma
\ref{Lefschetz}
$$[g^{n*}]=[1,\,1.12,\,1,\,-1.8].$$ Assume that $n=2n'$.
Then $g^{n'}$ is a non-symplectic automorphism of order $4$ such
that its square is symplectic. Since $\zeta_n^{n'}=-1$, we infer
that
$$[g^{n'*}]=[1,\, -1.12,\,\pm 1,\, (\zeta_4:2).4].$$
In any case, the Lefschetz fixed point formula (Proposition
\ref{trace}) yields $$e(g^{n'}):=e({\rm Fix}(g^{n'}))=\Tr(g^{n'*}|H^*(X)) < 0,$$ contradicting
Lemma \ref{nsym4}.
\end{proof}

\begin{lemma}\label{pn=10} Assume that $\phi(n)=10$. Then $m=1$ if $n=22$ and $m\le 2$ if $n=11$.
\end{lemma}

\begin{proof}
Since $[1,\,\zeta_n:\phi(n)]\subset[g^{*}]$,
$[1,\,1.10]\subset[g^{n*}].$ Since $g^{n}$ is symplectic of order
$m$, $m\le 2$ by Lemma \ref{Lefschetz}.

Assume that $m=2$ and $n=22$. Then $g^{11}$ is non-symplectic of order 4 with a symplectic square. Since
$$[g^{22*}]=[1,\, 1.10,\,-1.8,\,1.3],$$
we infer that $$[g^{11*}]=[1,\, -1.10,\,(\zeta_4:2).4,\,\pm 1,\,\pm 1,\,\pm 1].$$
In any case, $\Tr (g^{11*}|H^*(X))< 0$,  contradicting Lemma \ref{nsym4}.
\end{proof}

\begin{lemma}\label{pn=8} Assume $\phi(n)=8$.
\begin{enumerate}
\item  If $n=30,\, 24,\, 20$ or $16$, then $m=1$.
\item If $n=15$, then $m\le 2$.
\end{enumerate}
\end{lemma}

\begin{proof} Since $[1,\,\zeta_n:\phi(n)]\subset[g^{*}]$, $[1,\,1.8]\subset[g^{n*}].$ Since $g^{n}$
is symplectic of order $m$, $m\le 3$ by Lemma \ref{Lefschetz}.

\medskip Claim:  ${\rm ord}(g)\neq 2.16$.\\ On the contrary,
suppose that ${\rm ord}(g)= 2.16$. Then $-1\in [g^{16*}]$. Thus
$\zeta_{32}\in [g^*]$, then $[-1.\phi(32)]=[-1.16]\subset
[g^{16*}]=[1,\,1.13,\,-1.8]$.

\medskip Claim:  ${\rm ord}(g)\neq 3.16$.\\ Suppose that ${\rm
ord}(g)=3.16$. Then $[g^{16*}]=[1,\, 1.9,\, (\zeta_3:2).6]$. We
infer that
$$[g^{*}]=[1,\, \zeta_{16}:8,\, \pm 1,\, \eta_1,\ldots, \eta_{12}]$$
where $\eta_1,\ldots, \eta_{12}$ is a combination of $\zeta_3:2$,
$\zeta_6:2$, $\zeta_{12}:4$ and $\zeta_{24}:8$. In any case,
$$[g^{8*}]=[1,\, -1.8,\, 1,\, (\zeta_3:2).6],\,\,{\rm so}\,\,\Tr
(g^{8*}|H^*(X))<0,$$ contradicting Lemma \ref{nsym4}, as $g^{16}$ is symplectic.

\medskip Claim:  ${\rm ord}(g)\neq 3.15$.\\ Suppose that ${\rm
ord}(g)=3.15$. Then $[g^{15*}]=[1,\, 1.9,\, (\zeta_3:2).6].$ Since
$\phi(45)>12$, $\zeta_{45}\notin [g^*]$ and we infer that
$$[g^{*}]=[1,\, \zeta_{15}:8,\,1,\,(\zeta_9:6).2].$$
Then we compute $$[g^{3*}]=[1,\, (\zeta_{5}:4).2,\,1,\,(\zeta_3:2).6],\quad \Tr (g^{3*}|H^*(X))<0$$  contradicting Lemma \ref{nsym4}, as $g^{15}$ is symplectic.

\medskip Claim:  ${\rm ord}(g)\neq 3.30$.\\ Suppose that ${\rm
ord}(g)=3.30$. Then ${\rm ord}(g^2)=3.15$.

\medskip
Claim:  ${\rm ord}(g)\neq 3.24$.\\
Suppose that ${\rm ord}(g)=3.24$. Then $[g^{24*}]=[1,\, 1.8,
\,1,\, (\zeta_3:2).6]$, thus
$$[g^*]=[1,\, \zeta_{24}:8,\,\pm 1,\, \eta_1,\ldots, \eta_{12}]$$ where $[\eta_1,\ldots, \eta_{12}]$ is a combination of
 $\zeta_9:6$, $\zeta_{18}:6$ and
$\zeta_{36}:12$. The 4th power of $\zeta_9$, $\zeta_{18}$,
$\zeta_{36}$ is a 9th root of unity, so we infer that
$$[g^{4*}]=[1,\, (\zeta_6:2).4,\,1,\, (\zeta_9:6).2].$$ Thus
$$\Tr (g^{4*}|H^*(X))=8> \Tr (g^{24*}|H^*(X))=6,$$  contradicting Lemma \ref{nsym4}.

\medskip
Claim:  ${\rm ord}(g)\neq 2.24$.\\
Suppose that ${\rm ord}(g)=2.24$. Then $[g^{24*}]=[1,\,
1.8,\,1.5,\, -1.8]$. We infer that
$$[g^{*}]=[1,\, \zeta_{24}:8,\, \eta_1,\ldots, \eta_{5}, \,\zeta_{16}:8]$$
where $\eta_1,\ldots, \eta_{5}$ is a combination of 1, $-1$,
$\zeta_3:2$, $\zeta_4:2$, $\zeta_6:2$, $\zeta_8:4$ and
$\zeta_{12}:4$. In any case, $\sum_j\eta_j^8\le 5$ and $$e(g^8)=\Tr
(g^{8*}|H^*(X))=2+1-4+\sum\eta_j^8-8\le -4<0,$$  contradicting Lemma \ref{nsym4}, as $g^{24}$ is symplectic.

\medskip
Claim:  ${\rm ord}(g)\neq 3.20$.\\
Suppose that ${\rm ord}(g)=3.20$. Then $[g^{20*}]=[1,\, 1.8,\,1,\,
(\zeta_3;2).6]$, thus
$$[g^*]=[1,\, \zeta_{20}:8,\,\pm 1,\, \eta_1,\ldots, \eta_{12}]$$
where $\eta_1,\ldots, \eta_{12}$  is a combination of
$\zeta_{3}:2$, $\zeta_{6}:2$, $\zeta_{12}:4$, $\zeta_{15}:8$,
$\zeta_{30}:8$. We claim that $[\eta_1,\ldots,
\eta_{12}]=[(\zeta_{12}:4).3].$ Otherwise, $\sum_j\eta_j^{10}\le
2$ and
$$e(g^{10})=\Tr (g^{10*}|H^*(X))=2+1-8+1+\sum\eta_j^{10}\le -2<0,$$
 contradicting Lemma \ref{nsym4}, as $g^{20}$ is symplectic.
 This proves the claim. Thus
$[g^*]=[1,\, \zeta_{20}:8,\,\pm 1,\, (\zeta_{12}:4).3].$ Then
$$[g^{4*}]=[1,\, (\zeta_5:4).2, \, 1,\,(\zeta_3;2).6],\quad \Tr (g^{4*}|H^*(X))<0,$$  contradicting Lemma \ref{nsym4}.

\medskip
Claim:  ${\rm ord}(g)\neq 2.20$.\\
Suppose that ${\rm ord}(g)=2.20$. Then $[g^{20*}]=[1,\,
1.8,\,1.5,\, -1.8]$, thus
$$[g^{*}]=[1,\, \zeta_{20}:8,\, \eta_1,\ldots, \eta_{5}, \,(\zeta_{8}:4).2]$$
where $\eta_1,\ldots, \eta_{5}$ is a combination of 1, $-1$,
$\zeta_4:2$, $\zeta_5:4$ and $\zeta_{10}:4$. In any case,
$\sum_j\eta_j^{4}\le 5$ and $e(g^4)=\Tr (g^{4*}|H^*(X))=
1+\sum\eta_j^{4}-8<0$,  contradicting Lemma \ref{nsym4}, as $g^{20}$ is symplectic.

\medskip
Claim:  ${\rm ord}(g)\neq 2.30$.\\
Suppose that ${\rm ord}(g)=2.30$. Then $[g^{30*}]=[1,\,
1.8,\,1.5,\, -1.8]$. We infer that
$$[g^{*}]=[1,\, \zeta_{30}:8,\,\eta_1,\ldots, \eta_5,\, \tau_1,\ldots,\tau_8]$$
 where $[\eta_1,\ldots, \eta_5]$ is a combination of
$\pm 1$, $\pm\zeta_{3}:2$, $\pm\zeta_{5}:4$, and $[\tau_1,\ldots,\tau_8]$ is a combination of
$\zeta_{4}:2$, $\zeta_{12}:4$ and $\zeta_{20}:8$. In any case, we
see that
$$[g^{15*}]=[1,\, -1.8,\,\eta_1^{15},\ldots, \eta_5^{15},\, (\zeta_4:2).4].$$
Since $\eta_j^{15}=\pm 1$, we see that $\sum\eta_j^{15}\le 5$ and
$$0\le e(g^{15})=\Tr (g^{15*}|H^*(X))=-5+\sum\eta_j^{15}\le
0.$$ Thus $\eta_1^{15}=\ldots=\eta_5^{15}=1$. This occurs iff
$[\eta_1,\ldots, \eta_5]$ is a
combination of $1$, $\zeta_{3}:2$ and $\zeta_{5}:4$.
If $[\eta_1,\ldots, \eta_5]$ is a combination of $1$ and
$\zeta_{3}:2$, then using Lemma \ref{sum} we compute
$e(g^3)=\Tr (g^{3*}|H^*(X))=10> e(g^{30})=8.$ 
If $[\eta_1,\ldots, \eta_5]=[\zeta_{5}:4,\,1]$, then
$e(g^5)=\Tr (g^{5*}|H^*(X))=12> e(g^{30})=8.$ Both  contradict Lemma \ref{nsym4}, as $g^{30}$ is symplectic.
\end{proof}

\begin{lemma}\label{pn=6} Assume that $\phi(n)=6$.
\begin{enumerate}
\item If $n=18$, then $m=1$.
\item If $n=9$, then $m\le 2$.
\item If $n=14$, then $m=1$ or $3$.
\item If $n=7$, then $m\le 3$.
\end{enumerate}
\end{lemma}

\begin{proof}
We see that $g^{n}$ is symplectic of order $m$ with
$[1,\,1.6]\subset[g^{n*}]$, thus $m\le 4$ by Lemma
\ref{Lefschetz}.

Assume that $m=3$. Then $[g^{n*}]=[1,\,1.9,\,(\zeta_3:2).6]$. If
$n=9$ or 18, then $\zeta_{27}$ or $\zeta_{54}\in [g^*]$, but
$\phi(27)=\phi(54)=18>12$.

It remains to show that $2.18$, $4.9$, $4.7$ and $2.14$ do not
occur.

\medskip Claim: ${\rm ord}(g)\neq 2.18$.\\ Suppose
that ${\rm ord}(g)=2.18$.  Then $[g^{18*}]=[1,\, 1.6,\, 1.7,\,
-1.8]$, thus
$$[g^{*}]=[1,\, \zeta_{18}:6,\,\eta_1,\ldots,\eta_7,\,\tau_1,\,\ldots,\tau_8]$$
where $[\eta_1,\ldots,\eta_7]$ is a combination of $\pm 1$,
$\pm\zeta_3:2$, $\pm\zeta_9:6$ and
$[\tau_1,\ldots,\tau_8]=[(\zeta_{4}:2).4]$,
$[(\zeta_{4}:2).2,\,\zeta_{12}:4]$ or $[(\zeta_{12}:4).2]$. Then
$$[g^{6*}]=[1,\, (\zeta_{3}:2).3,\, \eta_1^6,\ldots,\eta_7^6,\, -1.8]$$
where $[\eta_1^6,\ldots,\eta_7^6]=[1.7]$ or $[1,\,
(\zeta_{3}:2).3]$. Thus, $\Tr (g^{6*}|H^*(X))<0$,  contradicting Lemma \ref{nsym4}, as $g^{18}$ is symplectic.

\medskip Claim: ${\rm ord}(g)\neq 4.9$.\\ Suppose
that ${\rm ord}(g)=4.9$.  Then $[g^{9*}]=[1,\, 1.6,\, 1,\,
(\zeta_4:2).4,\, -1.6]$, thus
$$[g^{*}]=[1,\, \zeta_{9}:6,\,1,\,\eta_1,\ldots,\eta_8,\,\tau_1,\,\ldots,\tau_6]$$
where $[\eta_1,\ldots,\eta_8]=[(\zeta_{4}:2).4]$,
$[(\zeta_{4}:2).2,\,\zeta_{12}:4]$ or $[(\zeta_{12}:4).2]$, and
$[\tau_1,\ldots,\tau_6]$ consists of $-1$, $\zeta_6:2$,
$\zeta_{18}:6$. Since $g^9$ is symplectic, by Lemma \ref{nsym4}
$\Tr (g^{3*}|H^*(X))\ge 0$. This determines
$[\tau_1,\ldots,\tau_6]$ uniquely,
$[\tau_1,\ldots,\tau_6]=[\zeta_{18}:6]$. Then
$$[g^{6*}]=[1,\, (\zeta_3:2).3,\,1,\, -1.8,\, (\zeta_3:2).3],\quad \Tr (g^{6*}|H^*(X))<0,$$  contradicting Lemma \ref{nsym4}, as $g^{18}$ is symplectic.

\medskip Claim:  ${\rm ord}(g)\neq 4.7$.\\ Suppose that ${\rm
ord}(g)=4.7$. Then $[g^{7*}]=[1,\, 1.6,\, 1,\, (\zeta_4:2).4,\,
-1.6]$, thus
$$[g^{*}]=[1,\, \zeta_{7}:6,\,1,\,(\zeta_4:2).4,\,\tau_1,\,\ldots,\tau_6]$$
where $[\tau_1,\ldots,\tau_6]=[\zeta_{14}:6]$ or $[-1.6]$. In
the second case, $\Tr (g^{*}|H^*(X))<0$, and in the first, $[g^{2*}]=[1,\, \zeta_{7}:6,\, 1,\,
-1.8,\,\zeta_7:6]$ and $\Tr (g^{2*}|H^*(X))<0$,  contradicting Lemma \ref{nsym4}, as $g^{14}$ is symplectic.

\medskip Claim:  ${\rm ord}(g)\neq 2.14$.\\ Suppose that ${\rm
ord}(g)=2.14$. Then $[g^{14*}]=[1,\, 1.6,\, 1.7,\, -1.8]$, thus
$$[g^{*}]=[1,\, \zeta_{14}:6,\,\eta_1,\ldots,\eta_7,\,(\zeta_4:2).4]$$
where $[\eta_1,\ldots,\eta_7]=[\zeta_{14}:6, \,\pm 1]$ or
$[\zeta_{7}:6, \,\pm 1]$ or $[\pm 1,\,\ldots, \pm 1]$. Since $g^{14}$ is symplectic, $\Tr (g^{2*}|H^*(X))\ge 0$ by Lemma \ref{nsym4}.
This is possible only if
$[\eta_1,\ldots,\eta_7]=[\pm 1,\,\ldots, \pm 1]$. Then
$\Tr (g^{2*}|H^*(X))=1$ and
${\rm Fix}(g^2)$ consists of a point. Thus the action of $g$ on
the 8-point set ${\rm Fix}(g^{14})$  fixes one  point and rotates
 7 points, hence $g^7$ fixes all the 8 points of ${\rm
Fix}(g^{14})$ and $$e({\rm
Fix}(g^{7}))=8.$$ On the other hand, $[g^{7*}]=[1,\, -1.6,\,
\eta_1^7,\ldots,\eta_7^7,\,(\zeta_4:2).4],$ thus $$\Tr (g^{7*}|H^*(X))= -3+\sum\eta_j^7\le 4.$$
\end{proof}

\begin{lemma}\label{pn=4} Assume that $\phi(n)=4$.
\begin{enumerate}
\item If $n=12$, then $m\le 3$ or $m=5$.
\item If $n=10$, then $m\le 3$.
\item If $n=8$, then $m\le 3$ or $m=5$.
\item If $n=5$, then $m\le 4$.
\end{enumerate}
\end{lemma}

\begin{proof} We see that $g^{n}$
is symplectic of order $m$ with $[1,\,1.4]\subset[g^{n*}]$, thus
$m\le 6$ by Lemma \ref{Lefschetz}.

Assume that $m=5$. Then $[g^{n*}]=[1,\,1.5,\,(\zeta_5:4).4]$. If
$n=5$ or 10, then $\zeta_{25}$ or $\zeta_{50}\in [g^*]$, but
$\phi(25)=\phi(50)=20>16$.

\medskip Assume that ${\rm ord}(g)=6.12$. Then ${\rm
ord}(g^2)=6.6$, but such an order does not occur by Lemma
\ref{pn=2}.

\medskip Assume that ${\rm ord}(g)=4.12$. Then ${\rm
ord}(g^3)=4.4$, but such an order does not occur by Lemma
\ref{pn=2}.

\medskip Assume that ${\rm ord}(g)=6.5$. Then\\ $[g^{5*}]=[1,\,
1.4,\, 1,\, (\zeta_3:2).4,\,(\zeta_6:2).2,\,-1.4]$, thus
$$[g^{*}]=[1,\,\zeta_{5}:4,\, 1,\, \eta_1,\,\ldots,\,\eta_8,\, (\zeta_6:2).2,\,\tau_1,\ldots,\tau_4]$$
where  $[\eta_1,\,\ldots,\,\eta_8]=[(\zeta_3:2).4]$ or
$[\zeta_{15}:8]$, and $[\tau_1,\ldots,\tau_4]=[\zeta_{10}:4]$ or
$[-1.4]$. By Lemma \ref{nsym4}, $0\le e(g)= \Tr (g^{*}|H^*(X))\le e(g^{5})=2$. Thus
$$[g^{*}]=[1,\,\zeta_{5}:4,\, 1,\,(\zeta_3:2).4,\,
(\zeta_6:2).2,\,\zeta_{10}:4]\,\,{\rm or}$$ $$[1,\,\zeta_{5}:4,\,
1,\,\zeta_{15}:8,\, (\zeta_6:2).2,\,-1.4].$$ In the first case,
$\Tr (g^{2*}|H^*(X))<0$, and in the second, $\Tr (g^{3*}|H^*(X))<0$,   contradicting Lemma \ref{nsym4}, as both $g^{10}$ and $g^{15}$ are symplectic.

\medskip Assume that ${\rm ord}(g)=4.10$. Then $[g^{10*}]=[1,\,
1.7,\, (\zeta_4:2).4,\,-1.6]$, so
$$[g^{*}]=[1,\,\zeta_{10}:4,\, \pm 1,\, \pm 1,\, \pm 1,\, (\zeta_8:4).2,\,(\zeta_4:2).3].$$
Thus $\Tr (g^{2*}|H^*(X))<0$, contradicting Lemma \ref{nsym4}.

\medskip Assume that ${\rm ord}(g)=6.8$ or $4.8$. Then ${\rm
ord}(g^2)=6.4$ or $4.4$, but such orders do not occur by Lemma
\ref{pn=2}.
\end{proof}

\begin{lemma}\label{pn=2} Assume that $\phi(n)\le 2$.
\begin{enumerate}
\item If $n=6$, then $m\neq 6, 8$.
\item If $n=4$, then $m\neq 4, 6, 8$.
\item If $n=3$, then $m\neq 6$.
\item If $n=2$, then $m\neq 8$.
\end{enumerate}
\end{lemma}

\begin{proof}

\medskip Assume that $\phi(n)= 2$ and $m=6$. Then $$[g^{n*}]=[1,\,
1.4,\, 1,\, (\zeta_3:2).4,\,(\zeta_6:2).2,\,-1.4].$$ We infer that
$\zeta_6\in[g^{n*}]$ must come from $\zeta_{6n}$ in $[g^*]$. But
$\phi(6n)=12, 8, 6$ respectively if $n=6, 4,3$. In any case,
$\phi(6n)>4$.

\medskip Assume that ${\rm ord}(g)=8.2$. Then $[g^{2*}]=[1,\,
1.3,\, (\zeta_8:4).2,\, (\zeta_4:2).3,\,-1.4].$ We infer that
$(\zeta_4:2).3$ must come from $\zeta_{8}:4$ in $[g^*]$, a
contradiction. This also shows that $8.4$ and $8.6$ do not occur.

\medskip Assume that ${\rm ord}(g)=4.4$. Then $[g^{4*}]=[1,\,
1.4,\, 1.3,\, (\zeta_4:2).4,\,-1.6].$ We infer that $-1.6$ must
come from $\zeta_{8}:4$ in $[g^*]$, a contradiction.
\end{proof}

%We have proved that $\phi(N)\le 20$, completing the proof of Proposition \ref{tame}. Now Theorem \ref{main-tame} follows from Proposition \ref{tame} and Proposition \ref{exist}.

\section{The Complex Case}

Throughout this section, $X$ is a complex K3 surface.

Assume that $X$ is not projective. Then  all automorphisms of
finite order are symplectic, hence of order $\le 8$, (see Section
3.)

Thus, we may assume that $X$ is projective. Then  the image of
${\rm Aut}(X)$ on $\GL(H^0(X, \Omega^2_X))\cong\bbC^*$ is a finite
cyclic group \cite{Nik}, and every automorphism of finite order
has an invariant ample divisor class, hence its induced action on
the 2nd integral cohomology $H^2(X,\bbZ)$ has 1 as an eigenvalue.
The proof of the only-if-part of Theorem \ref{main-complex} is
just a copy of the tame case, once we replace $H^2_{\rm
et}(X,\bbQ_l)$ by $H^2(X,\bbZ)$. Here, we also replace Proposition
\ref{trace} by the usual Lefschetz fixed point formula from
topology. Proposition \ref{C-exist} gives the if-part of Theorem
\ref{main-complex}.

We remark that in the complex case the holomorphic Lefschetz
formula improves Lemma \ref{nsym4} to a finer form: $X^h$ is
empty. But we do not need this.

\section{Faithfulness of the representation on the cohomology}

In this section we prove Theorem \ref{faith}.

\begin{proof}
Let $g$ be an automorphism of $X$ such that $g^*$ acts on
$H^2_{\rm et}(X,{\bbQ}_l)$ trivially. By Proposition
\ref{integral}, $g^*$ acts trivially on  $H^0(X,\Omega_X^2)$,
hence $g$ is symplectic. From the exact sequence of
$\bbQ_l$-vector spaces
\begin{equation}
0\to \NS(X)\otimes \bbQ_l \to  H_{\rm et}^2(X,\bbQ_l) \to
T_l^2(X)\to 0,
\end{equation}
we see that $g^*$ acts trivially on the N\'eron-Severi group
$\NS(X)$. As was pointed out by Rudakov and Shafarevich
(\cite{RS}, Sec. 8, Prop. 3), we may assume that $g$ is of finite
order. Indeed, automorphisms acting trivially on ${\rm Pic}(X)$
form an algebraic group and there is no non-zero regular vector
field on a K3 surface (\cite{RS}, Sec. 6, Theorem).

\medskip Case 1: ${\rm ord}(g)$ is coprime to the characteristic
$p$, i.e. $g$ is tame. By Proposition \ref{sym}, the
representation of a tame symplectic automorphism on  $H^2_{\rm
et}(X,{\bbQ}_l)$ is faithful. 

\medskip Case 2: ${\rm ord}(g)=p^rm$ for some $m$ coprime to $p$.
Since $g$  acts trivially on $H^2_{\rm et}(X,{\bbQ}_l)$, so does
$g^{p^r}$, which has order $m$. Then by Case 1, $g^{p^r}=1$.

\medskip Case 3: ${\rm ord}(g)=p^r$. Since $g$  acts trivially on
$H^2_{\rm et}(X,{\bbQ}_l)$, so is $g^{p^{r-1}}$. Thus we may
assume that $${\rm ord}(g)=p.$$ By Theorem 2.1 \cite{DK2}, an
automorphism of order the characteristic  $p$ exists only if $p\le
11.$ By the result of \cite{DK1} we know that the quotient surface
$X/\langle g\rangle$ is a rational surface or an Enriques surface
(non-classical of $\mu_2$-type) or a K3 surface with rational
double points.

\medskip
Assume that $X/\langle g\rangle$ is a rational surface. This case
occurs when the fixed locus $X^g$ is either a point giving a
Gorenstein elliptic singularity on the quotient surface or of
1-dimensional. In the latter case the quotient surface has
rational singularities. Since  $X/\langle g\rangle$ is rational,
by Proposition \ref{diminv}
$$\dim H^2_{\rm et}(X,{\bbQ}_l)^g=\rank~\NS(X)^g.$$
Since $g^*$ acts trivially on  $H^2_{\rm et}(X,{\bbQ}_l)$, it acts
trivially on
 $\NS(X)$. Thus
$$\rank~\NS(X)=\rank~\NS(X)^g=\dim H^2_{\rm et}(X,{\bbQ}_l)^g=22,$$
i.e. $X$ is supersingular. For supersingular K3 surfaces Rudakov and Shafarevich
(\cite{RS}, Sec. 8, Prop. 3) proved the faithfulness of the automorphism group on the N\'eron-Severi group.

\medskip
Assume that $X/\langle g\rangle$ is an Enriques surface. This case
occurs when $p=2$ and $X^g$ is empty. By Proposition \ref{diminv},
we have
$$\dim H^2_{\rm et}(X,{\bbQ}_l)^g=\rank~\NS(X)^g=10,$$
thus $g$ cannot act trivially on $\dim H^2_{\rm et}(X,{\bbQ}_l)$. 

\medskip
Assume that $X/\langle g\rangle$ is a K3 surface with rational
double points. This case occurs when the fixed locus $X^g$
consists of either two points or a point which gives a rational
double point on the quotient.
  Let $Y\to X/\langle g\rangle$
be a minimal resolution. Then $Y$ is a K3 surface and
$$\rank~\NS(Y)>\rank~\NS(X/\langle g\rangle)=\rank~\NS(X)^g.$$
On the other hand, by Proposition 5 of \cite{Shioda2}
$$T_l^2(Y)\cong T_l^2(X)^g.$$ Combining these two, we get
$$\dim H^2_{\rm et}(Y,{\bbQ}_l)=\rank~\NS(Y)+ \dim T_l^2(Y)>\rank~\NS(X)^g + \dim T_l^2(X)^g.$$
The right hand side is 22, if $g^*|H^2_{\rm
et}(X,{\bbQ}_l)$ is trivial.
\end{proof}

\section{The case: $p=11$}

In this section we assume that $p=11$ and determine the orders of
wild automorphisms. We first recall some previous results from
\cite{DK1}, \cite{DK2}, \cite{DK3}.

\begin{proposition}\label{11} \cite{DK1}, \cite{DK2}, \cite{DK3}  Let $u$ be an automorphism of order $11$ of a K3 surface $X$ over an algebraically closed
field of characteristic $p=11$. Then
\begin{enumerate}
\item the fixed locus $X^u$ is a cuspidal curve or a point;
\item $X$ admits an $u$-invariant fibration $X\to {\bf P}^1$ of curves of genus $1$, with or without a section, and $u$ acts on the base ${\bf P}^1$ with one
fixed point;
\item the corresponding fibre $F_0$ is of type $II$ and the fixed locus $X^u$ is either equal to $F_0$ or the cusp of
$F_0$;
\item the action of $u^*$ on $H^2_{\rm
  et}(X,\bbQ_l)$, $l\ne 11$, has eigenvalues $$[u^*]=[1,\,1,\, (\zeta_{11}:10).2]$$
  where the first eigenvalue corresponds to a $u$-invariant ample class.
  \end{enumerate}
\end{proposition}

\begin{proof} (1)--(3) are contained in \cite{DK1} and \cite{DK2}.

(4) By Proposition 4.2 \cite{DK3}, $\Tr (u^*|H^2_{\rm
et}(X,{\bbQ}_l))=0.$ Since an eigenvalue of $u^*$ is  an $11$-th
root of unity, the result follows.
\end{proof}

We now state the main result of this section.

\begin{theorem} \label{11main} Let $g$ be an automorphism of finite order of a K3 surface $X$ over an algebraically closed
field of characteristic $p=11$. Assume that the order of $g$ is
divisible by $11$. Then $${\rm ord}(g) = 11n, \,\,\,{\rm
where}\,\,\,n=1,\,2,\,3,\,6.$$ All these orders are realized by an
example (see Example \ref{exam11}).
\end{theorem}

The proof follows from the following two lemmas \ref{11q} and
\ref{44}.

\begin{lemma}\label{11q} $n\neq 9$ and $n$ is not a prime $>3$.
\end{lemma}

\begin{proof} Suppose that $n=9$ or a prime $\ge 5$. Then $g^{n*}$ on $H^2_{\rm
  et}(X,\bbQ_l)$, $l\ne 11$, has eigenvalues $[g^{n*}]=[1,\,1,\, (\zeta_{11}:10).2]$ where the first eigenvalue corresponds to an invariant ample divisor.
Since $\phi(11n)>20$, the eigenvalues $\zeta_{11n}:\phi(11n)$
cannot appear in $[g^{*}]$. Then by the faithfulness of Ogus (Theorem \ref{faith}), the
eigenvalues $\zeta_{n}:\phi(n)$ must appear in $[g^{*}]$. But
$\phi(n)>1$.
\end{proof}

\begin{lemma}\label{44} $n\neq 4$.
\end{lemma}

\begin{proof} Suppose that $n=4$.
By by the faithfulness (Theorem \ref{faith}),
$g^*$ acting on $H^2_{\rm et}(X,{\bbQ}_l)$ has $\zeta_{44}$ as an eigenvalue and hence
$$[g^*]=[1,\,\pm 1,\, \zeta_{44}:20]$$ where the first eigenvalue corresponds to a $g$-invariant ample divisor.
Let
$$s:=g^{11},\qquad u:=g^4.$$ Then $s^2$ is an involution of $X$ with $[s^{2*}]=[1,\,1,\,-1.20].$ Thus by Lemma \ref{nsym2}, $s^2$ is non-symplectic 
and  $X/\langle
s^2 \rangle\cong {\bf F}_e$ a rational ruled surface.
Applying Deligne-Lusztig (Proposition \ref{DL}) to $g^{26}=s^2u$, we see that
$$\Tr (u^*|H^*(X^{s^2}))=\Tr (g^{26*}|H^*(X))=6.$$
Assume that ${\rm Fix}(s^2)$ is a curve $C_9$ of genus 9.
 The characteristic polynomial of $u^*|H^1_{\rm et}(C_9,{\bbQ}_l)$
has integer coefficients, hence $u^*|H^1_{\rm et}(C_9,{\bbQ}_l)$ has trace 18 or 7, so
$$\Tr (u^*|H^*(X^{s^2}))=-16\,\, {\rm or}\, -5,$$  neither
compatible with the above computation. Thus $${\rm
Fix}(s^2)=C_0\cup C_{10}$$ where $C_i$ is a curve of genus $i$. In
this case Deligne-Lusztig does not work, as  $u^*|H^1_{\rm
et}(C_{10},{\bbQ}_l)$ may have trace $-2$. We need a new argument.
Note that the induced automorphism $\bar{g}$ leaves invariant the
unique ruling of ${\bf F}_4$. Let
$$\psi:X\to {\bf P}^1$$
be the fibration of curves of genus 1 induced from the ruling on
the quotient. It follows that $g$ acts on the base ${\bf P}^1$, and
$C_0$ is a section of $\psi$ and $C_{10}$ a 3-section. Clearly
$g^{44}=1$ on ${\bf P}^1$. The order 11 automorphism $u=g^4$ of $X$
acts nontrivially on the base ${\bf P}^1$; otherwise it would be
induced by the translation by an 11-torsion, but there is no
$p$-torsion on an elliptic K3 surface in characteristic $p>7$
(Theorem 2.13 \cite{DK2}). By Lemma \ref{P1}, $s=g^{11}$ acts
trivially on the base ${\bf P}^1$. Thus we infer that $s=g^{11}$
acts trivially on both $C_0$ and $C_{10}$, i.e.,
$$X^{s}=C_0\cup C_{10}.$$ Then $e(X^{s})=-16$. On the other hand, from the list $[g^*]$ we compute that
$[s^*]=[g^{11*}]=[1,\,1,\,(\zeta_4:2).10],$ hence $\Tr (s^{*}|H^*(X))=4\neq e(X^{s})$.
\end{proof}

\begin{example}\label{exam11}
In char $p = 11$, there are K3 surfaces with an automorphism of
order 22
$$X_{\varepsilon}: y^2+x^3+\varepsilon x^2+t^{11}-t = 0, \qquad g_{\varepsilon, 22}(t,x,y)=(t+1, x, -y).$$
When $\varepsilon=0$, $X_0$ also admits an automorphism of order
66
$$g_{66}:(t,x,y) \mapsto (t+1,\zeta_3 x,-y).$$
The surface $X_{\varepsilon}$ has a
 $II$-fibre at $t=\infty$ and 22 $I_1$-fibres at $(t^{11}-t)(t^{11}-t+3\varepsilon^3)=0$ if $\varepsilon\neq 0$,  and  11 $II$-fibres at
 $t^{11}-t=0$ if $\varepsilon=0$ (\cite{DK3}, \cite{DK1}, 5.8).
\end{example}

In characteristic 11, it is known (Lemma 2.3 \cite{DK3}) that a K3
surface admitting an automorphism of order $11$ has Picard number
$2$, $12$ or $22$. For K3 surfaces with an automorphism of order
33, the second cannot occur.

\begin{proposition}\label{33} In characteristic $p=11$, if $X$ admits an automorphism of order
$33$, then the Picard number $\rho(X)=2$ or $22$.
\end{proposition}

\begin{proof} Let $g$ be an automorphism of order
$33$. Our method shows that the action of $g^*$ on $H^2_{\rm
  et}(X,\bbQ_l)$, $l\ne 11$, has eigenvalues
  $[g^*]=[1,\, 1,\, \zeta_{33}:20].$
On the other hand, by Proposition \ref{diminv} both traces of
$g^*$ on ${\rm NS}(X)$ and on $T_l^2(X)$ are integers. Since
$X/\langle g\rangle$ is a rational surface, $\dim T_l^2(X)^g=0$.
\end{proof}

\section{The case: $p=7$}

In this section we determine the orders of wild automorphisms in
characteristic $p=7$. We first improve the previous results from
\cite{DK1} and \cite{DK2}.

\begin{proposition}\label{7} Let $u$ be an automorphism of order $7$ of a K3 surface $X$ over an algebraically closed
field of characteristic $p=7$. Then
\begin{enumerate}
\item the fixed locus $X^u$ is a connected curve or a point and $X/\langle u\rangle$ is a rational surface;
\item $X$ admits an $u$-invariant fibration $\psi: X\to {\bf P}^1$ of curves of arithmetic genus $1$, with or without a section, and there is a specific fibre $F_0$ such that $X^u$ is either the
support of $F_0$ or a point of $F_0$;
\item if  $X^u={\rm supp}(F_0)$, then $u$ acts on ${\bf P}^1$ non-trivially and $F_0$ is of type $II^*$ or $III$;
\item if  $X^u$  is a point, then $F_0$ is of type $III$,  $X^u$ is the singular point of $F_0$ and
\begin{enumerate}
\item $u$ acts on ${\bf P}^1$ non-trivially; or
\item $u$ acts on ${\bf P}^1$ trivially, $\psi$ has $3$ singular fibres of type $I_7$ away from $F_0$ and $u$ is induced by the translation by a $7$-torsion of the Jacobian fibration of
$\psi$;
\end{enumerate}
\item  $u^*|H^2_{\rm
  et}(X,\bbQ_l)$, $l\ne 7$, has eigenvalues
  $$[u^*]=[1,\, 1.9,\, (\zeta_7:6).2]\,\,\,or\,\,\,[1,\, 1.3,\, (\zeta_7:6).3]$$
  respectively if $F_0$  is of type $II^*$ or $III$, where the first eigenvalue corresponds to an u-invariant ample divisor.
\end{enumerate}
\end{proposition}

\begin{proof} The assertions (1)--(3) are explicitly stated in  \cite{DK1} and \cite{DK2}. We will prove the last two assertions.

(4) Assume that $X^u$ is a point. Let $F_0$ be the fibre of $\psi$
containing $X^u$.

\medskip
Assume that $u$ acts on ${\bf P}^1$ non-trivially. The singular
fibres other than $F_0$ form orbits under the action of $u$, hence
the sum of their Euler numbers is divisible by the characteristic.
If the sum $\le 14$, then $F_0$ has Euler number $e(F_0)\ge 10$.
But no singular fibre with Euler number $\ge 10$ admits an order 7
automorphism with just one fixed point. Thus the sum is $21$. It
follows that $e(F_0)=3$ and hence $F_0$ is of type $III$ or $I_3$.
Since a fibre of type $I_3$ does not admit an order 7 automorphism
with one fixed point,  $F_0$ is of type $III$ and $X^u$ is the
singular point of $F_0$.

\medskip
Assume that $u$ acts on ${\bf P}^1$ trivially. Then $u$ is induced by
the translation by a $7$-torsion of the Jacobian fibration of
$\psi$. Each singular fibre other than $F_0$ is acted on by $u$
without fixed points, so must be of type $I_{7m}$. Let
$$I_{7m_1}, \ldots, I_{7m_r}$$ be the types of all singular fibres
other than $F_0$. Then $$e(F_0)=24-7\sum m_i.$$ Since $F_0$ admits
an order 7  automorphism with one fixed point, this is possible
only if $\sum m_i=3$ and $F_0$ is of type $III$. To determine
$m_1\ldots, m_r$, we look at the Jacobian fibration. It is known
that the Jacobian fibration is again a K3 surface with singular
fibres of the same type as $\psi$ (see \cite{CD}). So, to
determine the types of singular fibres, we may assume that $\psi$
has a section and $u$ is the translation by a $7$-torsion of the
fibration $\psi$. Applying the explicit formula for the height
pairing (\cite{CoxZ} or Theorem 8.6 \cite{Shioda3}) on the
Mordell-Weil group of an elliptic surface, or in our K3 case the
formula given in \cite{DK1}, p.121,   we deduce that $r=3$ and
$m_1=m_2=m_3=1.$

\bigskip
(5) Assume that $F_0$ is of type $II^*$. Then the 9 components of
$F_0$ are fixed by $u$. The rational elliptic surface $X/\langle
u\rangle$ is singular along the fibre  $\bar{F_0}$ coming from
$F_0$, but the sum of Euler numbers of other singular fibres is 2.
We infer that  the fibre $Y_0$ of the relative minimal model $Y$
of the minimal resolution of $X/\langle u\rangle$ corresponding to
$\bar{F_0}$ must be of type $II^*$ or $I_4^*$. Since the number of
components of $Y_0$ is the number of components of  $\bar{F_0}$,
we see that $$\rank~{\rm NS}(X/\langle u\rangle)=10.$$ Then by
Proposition \ref{diminv}, $\dim H^2_{\rm
et}(X,{\bbQ}_l)^u=\rank~{\rm NS}(X)^u=10.$

Assume that $F_0$ is of type $III$. Then the order 7 automorphism
$u$ preserves each of the two components of $F_0$. The rational
elliptic surface $X/\langle u\rangle$ is singular along the fibre
$\bar{F_0}$ coming from $F_0$, but the sum of Euler numbers of
other singular fibres is 3.  We infer that the fibre $Y_0$ of the
relative minimal model $Y$ of the minimal resolution of $X/\langle
u\rangle$ corresponding to $\bar{F_0}$ must be of type $III^*$ or
$I_3^*$. Since the number of components of $Y_0$ is 8 while the
number of components of  $\bar{F_0}$ is 2, we see that
$$\rank~{\rm NS}(X/\langle u\rangle)=10-6=4.$$ Then  by
Proposition \ref{diminv}, $\dim H^2_{\rm
et}(X,{\bbQ}_l)^u=\rank~{\rm NS}(X)^u=4.$
\end{proof}

The following is the main result of this section.

\begin{theorem} \label{7main} Let $g$ be an automorphism of finite order of a K3 surface $X$ over an algebraically closed
field of characteristic $p=7$. Assume that the order of $g$ is
divisible by $7$. Then $${\rm ord}(g) = 7n, \,\,\,{\rm
where}\,\,\,n=1,\,2,\,3,\,4,\,6.$$ All these orders are realized
by an example (see Example \ref{exam7}).
\end{theorem}

The proof follows from the following lemmas \ref{7q}---\ref{84}.

\begin{lemma}\label{7q}   $n\neq 9$, $n\neq 8$ and $n\neq$  a prime $>3$.
\end{lemma}

\begin{proof} Consider the order 7 automorphism $$u:=g^{n}.$$ Suppose that $n=9$ or 8 or
 a prime $\ge 5$.

\bigskip
Case: $[u^*]=[1,\, 1.3,\, (\zeta_7:6).3]$. By the faithfulness of Ogus (Theorem \ref{faith}),
$g^*$ on $H^2_{\rm et}(X,\bbQ_l)$ should have $\zeta_n$ or
$\zeta_{7n}$ as an eigenvalue. If  $\zeta_n\in[g^*]$, then the
eigenvalues $1.\phi(n)$ appear in $[u^*]$, but $\phi(n)>3$. Since
$\phi(7n)>18$, $\zeta_{7n}\notin[g^*]$.

\bigskip
Case: $[u^*]=[1,\, 1.9,\, (\zeta_7:6).2]$. In this case  $X^u$ is
the support of a $II^*$ fibre $F_0$. Since $g$ acts on $X^u$, it
preserves every component of $F_0$. The 9 components and a
$g$-invariant ample class are
  linearly independent in the N\'eron-Severi group. Thus
  $\rank~\NS(X)^g= 10$. Then  by the faithfulness, $\zeta_{7n}\in[g^*]$, but $\phi(7n)>12$.
\end{proof}

\begin{lemma}\label{28} If $n=4$, then
$$[g^{4*}]=[1,\, 1.3,\, (\zeta_7:6).3].$$
\end{lemma}

\begin{proof}  Consider the order 7 automophism $$u:=g^{4}.$$  Suppose that $[u^*]=[1,\, 1.9,\, (\zeta_7:6).2]$ and
 ${\rm Fix}(u)$ is the support of a type $II^*$ fibre $F_0$ of a  fibration $\psi:X\to \bbP^1$. As in the proof of Lemma \ref{7q}, we infer that
 $\rank~\NS(X)^g= 10.$ By the faithfulness (Theorem \ref{faith}),
$$[g^*]=[1,\, 1.9,\, \zeta_{28}:12].$$
 Let
$$s:=g^{7}.$$  Then $s$ is a tame automorphism of order 4 with
$[s^*]=[1,\, 1.9,\, (\zeta_4:2).6].$ By Proposition \ref{trace},
$e(X^s)=12$. Thus $X^{s}$ is non-empty. The involution
$s^2=g^{14}$ has $[s^{2*}]=[g^{14*}]=[1,\, 1.9,\, -1.12]$, hence
$e(s^2)=0$.  Since ${\rm Fix}(s^2)$ contains   ${\rm
Fix}(s)$, it is non-empty. Let us determine ${\rm Fix}(s^2)$. Let
$R_1,\, R_2,\,\ldots, R_9$ be the components of $X^u$.

\bigskip
\bigskip
 $$
\begin{picture}(180,30)
\put(0,25){$R_8-R_7-R_6-R_5-R_4-R_3-R_2-R_1$}
 \put(60,15){\line(0,0){6}}
 \put(55,5){$R_9$}
\end{picture}
$$

\bigskip
\bigskip\noindent Note that $s^2$ acts on $X^u$ as an involution.
Any involution of ${\bf P}^1$ in char $\neq 2$ has exactly two fixed
points. The component $R_6$ is acted on by $s^2$ with 3 fixed
points, hence fixed  point-wisely by $s^2$. Note that  a
non-symplectic involution of a K3 surface has no isolated fixed
points. If $R_8$ is not fixed point-wisely by $s^2$, then a curve
$C'\subset{\rm Fix}(s^2)$ must pass through the intersection point
of $R_8$ and $R_7$, then locally at the intersection of the 3
curves $s^2$ preserves $R_8$ and $R_7$ and fixes $C'$ point-wisely,
which is impossible in any characteristic for a local tame
automorphism of a 2-dimensional space. Thus $R_8$ is fixed
point-wisely. Similarly, $R_4$ and $R_2$ are fixed point-wisely.
Then there is a curve $C\subset {\rm Fix}(s^2)$ intersecting $R_9$
and $R_1$ each with multiplicity 1. The curve $C$ is either
irreducible or a union of two curves respectively intersecting
$R_9$ and $R_1$. We claim that
$${\rm Fix}(s^2)=R_2\cup R_4\cup R_6 \cup R_8\cup C.$$
Indeed, if there is another component of $X^{s^2}$, then it does not meet $F_0$, hence must be contained in a fibre. But any fibre different from $F_0$ is irreducible (Proposition \ref{7}).  Computing $[g^{18*}]=[1,\,1.9,\,(\zeta_{14}:6).2]$ and applying  Deligne-Lusztig (Proposition \ref{DL}) to $g^{18}=s^2u$, we get
$$\Tr (u^*|H^*(X^{s^2}))=\Tr (g^{18*}|H^*(X))=14.$$
Assume that $C$ is irreducible. Then it is a curve $C_5$ of genus 5, since $e(s^2)=0$.
It is easy to see that
$\Tr (u^*|H^1_{\rm et}(C_5,{\bbQ}_l))=10$ or 3. Thus
$$\Tr (u^*|H^*(X^{s^2}))=0\,\, {\rm or}\,\, 7,$$ neither
compatible with the above computation. Thus $$C=C_0\cup C_6$$
where $C_i$ is a curve of genus $i$. In this case Deligne-Lusztig
does not work, as it may happen that $\Tr (u^*|H^1_{\rm
et}(C_6,{\bbQ}_l))=-2$. We employ a new argument. We infer that
$C_0$ meets $R_1$, hence is a section, and  $C_6$ meets $R_9$,
hence is a 3-section of $\psi:X\to {\bf P}^1$. It follows that $s^2$
acts on a general fibre of $\psi$ with 4 fixed points, the
intersection of $C$ and the fibre. Let $$Y\to X/\langle
s^2\rangle$$ be the minimal resolution. Then $Y$ is a rational
ruled surface, not relatively minimal. Note that $g(X^u)=X^u$ and
hence $g$ preserves the fibration $\psi:X\to {\bf P}^1$. Clearly
$g^{28}=1$ on the base ${\bf P}^1$. By Proposition \ref{7}, $g^4$
acts non-trivially on the base ${\bf P}^1$. Thus $s=g^{7}$  acts
trivially on the base ${\bf P}^1$ by Lemma \ref{P1}. On the other
hand, $s$ acts on ${\rm Fix}(s^2)$ and on $X^u$, hence on both
$C_0$ and $C_6$. Since $C_0$ is a section of $\psi$ and $C_6$ a
3-section, we see that $s$ acts trivially on both $C_0$ and $C_6$,
i.e., $$X^{s}\supset C_0\cup C_6.$$  Thus $s$ induces an
involution $\bar{s}$ on $Y$, which leaves invariant each fibre of
$Y$. Moreover, on a general fibre of $Y$ $\bar{s}$ has 4 fixed
points, but no involution of ${\bf P}^1$ can have 4 fixed points.
\end{proof}

\begin{lemma}\label{84} $n\neq 12$.
\end{lemma}

\begin{proof} Suppose that $n=12$.
Define $$u:=g^{12}.$$ Then $u$
has order 7. Applying Lemma \ref{28} to $g^3$, we see that
$$[u^*]=[1,\, 1.3,\, (\zeta_7:6).3].$$
Since $\phi(84)>18$, $\zeta_{84}\notin [g^*]$. By Theorem
\ref{faith}, we infer that the list of eigenvalues of $[g^*]$ is
one of the following:
$$[g^*]=[1,\,\pm 1,\,\pm(\zeta_3:2),\, \zeta_{28}:12,\,\pm(\zeta_7:6)],$$
$$[g^*]=[1,\,\pm 1,\,\zeta_4:2,\, \pm(\zeta_{21}:12),\,\pm(\zeta_7:6)],$$
where $\pm(\zeta_{21}:12)$ means $\zeta_{21}:12$ or $\zeta_{42}:12$. In any case, we claim that
$${\rm Fix}(u)={\rm Fix}(g^{12})= \,\,\{{\rm a\,\,point}\}.$$
Indeed, if ${\rm Fix}(u)$ is the support of a fibre $F_0$ of type
$III$, then the two components of $F_0$ are interchanged or
preserved by $g$, so the eigenvalues $[1,\,1,\,\pm 1]$ should
appear in $[g^*]$, impossible.

\medskip Assume the first case $[g^*]=[1,\,\pm
1,\,\pm(\zeta_3:2),\, \zeta_{28}:12,\,\pm(\zeta_7:6)]$. The tame
involution $s:=g^{42}$ has
$$[s^*]=[g^{42*}]=[1,\,1,\,1.2,\, -1.12,\,1.6],$$
hence it is non-symplectic and by Proposition \ref{trace} $e(g^{42})=0$. If  $g^{42}$ acts freely on $X$, then so does $g^{21}$.
But no K3 surface admits an order 4 free action. Thus ${\rm Fix}(g^{42})$ is either a union of elliptic curves or a union of a curve of genus $d+1\ge 2$ and $d$ smooth rational curves.
Applying Deligne-Lusztig (Proposition \ref{DL}) to $g^{54}=su$,
we get
$$\Tr (u^*|H^*(X^{s}))=\Tr (g^{54*}|H^*(X))=7.$$
Any order 7 or trivial action on an elliptic curve $E$ has
$$\Tr (u^*|H^*(E))=1-2+1=0.$$ If 7  elliptic curves form an orbit under $u$, then the trace on the $j$-th cohomology of the union of the 7 curves is 0.
This rules out the first possibility for ${\rm Fix}(g^{42})$. Thus ${\rm Fix}(g^{42})$ is a union of a curve $C_{d+1}$ of genus $d+1\ge 2$ and $d$ smooth rational curves.
Since fixed curves give linearly independent invariant vectors in $H^2_{\rm et}(X,{\bbQ}_l)$ and $\dim H^2_{\rm et}(X,{\bbQ}_l)^s=10$,
$$1\le d\le 9.$$
Consider the action of $u=g^{12}$ on  ${\rm Fix}(g^{42})$. It is
of order 7 or trivial. Since ${\rm Fix}(g^{12})$ is a point, the
action of $u$ on  ${\rm Fix}(g^{42})$ has at most one fixed point
and is of order 7. This is possible only if
 $d=7r$, $7r+1$,  and the $7r$ smooth rational curves form $r$ orbits under  $u$.
If $d=7$ and $u$ acts on $C_8$, then, since $\Tr
(u^{*}|H^1_{\rm et}(C_8,{\bbQ}_l))=16-7b$, $0\le b\le 2$, we see
that
$$\Tr (u^*|H^*(X^{s}))=\Tr (u^*|H^*(C_8))=2-(16-7b)=7b-14\neq 7.$$
If $d=7r+1$, then $\Tr (u^*|H^1(C_{d+1}))=2(7r+2)-7b$, $0\le b\le 2r$, and
$$\Tr (u^*|H^*(X^{s}))=2+ \Tr (u^*|H^*(C_{d+1}))=7b-14r\neq 7.$$

\medskip Assume the second case $[g^*]=[1,\,\pm 1,\,\zeta_4:2,\,
\pm(\zeta_{21}:12),\,\pm(\zeta_7:6)]$. The tame involution
$s:=g^{42}$ has
$$[s^*]=[g^{42*}]=[1,\,1,\,-1.2,\, 1.12,\,1.6],$$
hence it is non-symplectic and by Proposition \ref{trace} $e(g^{42})=20$. Thus ${\rm Fix}(g^{42})$ is either a union of
$10$ smooth rational curves and possibly some elliptic curves or a
union of a curve $C_{d-9}$ of genus $d-9\ge 2$ and $d$ smooth rational
curves. In the first case, the order 7 action of $g^{6}$ on ${\rm
Fix}(g^{42})$ preserves at least 3 smooth rational curves, hence
fixes at least 3 points, a contradiction. Thus we have the second
case. Since fixed curves give linearly independent invariant
vectors in $H^2_{\rm et}(X,{\bbQ}_l)$ and $\dim H^2_{\rm
et}(X,{\bbQ}_l)^s=20$,
$$11\le d\le 19.$$ Since ${\rm
Fix}(g^{6})\subset {\rm Fix}(g^{12})$, the action of $g^{6}$ on
${\rm Fix}(g^{42})$ has at most one fixed point and is of order 7.
This is possible only if $d=7r$ or $7r+1$ and the $7r$ smooth rational curves form $r$ orbits under  $g^{6}$.
In any case, $r=2$ and each orbit (and $C_{d-9}$ also) gives a $g^{6}$-invariant
divisor, hence $$\rank~\NS(X)^{g^{6}}\ge 3.$$ But $\dim H^2_{\rm
et}(X,{\bbQ}_l)^{g^{6}}=2$.
\end{proof}

\begin{example}\label{exam7} In char $p = 7$, there are K3 surfaces with an automorphism of order 42 or 28.

\medskip
(1) $X_{42}: y^2=x^3+t^{7}-t$, $g_{42}(t,x,y)=(t+1,\zeta_3 x, -y)$;

\medskip
(2) $X_{28}: y^2=x^3+(t^{7}-t)x$, $g_{28}(t,x,y)=(t+1, -x, \zeta_4 y)$.

\medskip\noindent
The surface $X_{42}$ has a
 $II^*$-fibre at $t=\infty$ and 7 $II$-fibres at $t^7-t=0$; $X_{28}$ has 8
 $III$-fibres at $t=\infty$, $t^7-t=0$
 (\cite{DK1}, 5.8).
\end{example}

\section{The case: $p=5$}

In this section we determine the orders of wild automorphisms in
characteristic $p=5$. We first improve the previous results from
\cite{DK1} and \cite{DK2}.

\begin{proposition}\label{5} Let $u$ be an automorphism of order $5$ of a K3 surface $X$ over an algebraically closed
field of characteristic $p=5$. Then one of the following occurs:
\begin{enumerate}
\item the fixed locus $X^u$ contains a curve of arithmetic genus
$2$;
\item $X^u$  is the
support of a fibre $F_0$ of a fibration $\psi: X\to {\bf P}^1$ of
curves of arithmetic genus $1$, with or without a section, $F_0$
is of type $IV$ or $III^*$, and $u^*|H^2_{\rm
  et}(X,\bbQ_l)$, $l\ne 5$, has eigenvalues
$$\qquad\quad [u^*]=[1,\, 1.5,\, (\zeta_5:4).4]\,\,{\rm or}\,\, [1,\, 1.9,\,
    (\zeta_5:4).3]$$
respectively if $F_0$ is of type $IV$ or $III^*$;
\item  $X^u$ consists of two points, $X/\langle u\rangle$ is a K3 surface with two rational double points of type $E_8$, and
 $u^*|H^2_{\rm et}(X,\bbQ_l)$ has eigenvalues
  $$[u^*]=[1,\, 1.5,\, (\zeta_5:4).4];$$
\item  $X^u$ consists of a point, $X/\langle u\rangle$ is a K3 surface with one rational double point of type $E_8$, $X$ does not admit an $u$-invariant elliptic fibration, and
 $u^*|H^2_{\rm
  et}(X,\bbQ_l)$ has eigenvalues
  $$[u^*]=[1,\, 1.13,\, (\zeta_5:4).2];$$
\item  $X^u$ consists of a point, $X/\langle u\rangle$ is a rational surface, $X$ admits an $u$-invariant elliptic fibration with a fibre $F_0$ of type $IV$ whose singular point is $X^u$, and
 $u^*|H^2_{\rm
  et}(X,\bbQ_l)$ has eigenvalues
  $$[u^*]=[1,\, 1.5,\, (\zeta_5:4).4]$$
\end{enumerate}
      where the first eigenvalue corresponds to an u-invariant ample divisor.
\end{proposition}

\begin{proof} By \cite{DK1}, the fixed locus $X^u$ consists of a point, two points or a connected curve of Kodaira dimension $\kappa(X, X^u)=0,1,2$. By Proposition 2.5 \cite{DK2},
the case with $\kappa(X, X^u)=0$ occurs only in characteristic 2,
so does not occur in characteristic 5.

\medskip Assume that $X^u$ is a connected curve of Kodaira
dimension $\kappa(X, X^u)=2$. Then by Proposition 2.3 \cite{DK2},
$X^u$ contains a curve of arithmetic genus 2. This gives the case
(1).

\medskip Assume that $X^u$ is a connected curve of Kodaira
dimension $\kappa(X, X^u)=1$. Then by Theorem 3 \cite{DK1}, $X^u$
is the support of a fibre $F_0$ of a fibration $$\psi: X\to
{\bf P}^1$$ of curves of arithmetic genus $1$, with or without a
section, the induced action of $u$ on the base ${\bf P}^1$ is of
order 5 and the quotient $X/\langle u\rangle$ is a (singular)
rational elliptic surface.  Let
$$Y\to X/\langle u\rangle$$ be a minimal resolution. Then $Y$ is a rational elliptic
surface, not necessarily relatively minimal.  Let $\overline{Y}$
be the relatively minimal rational elliptic surface. Denote by
$Y_0\subset Y$ and $\overline{Y_0}\subset\overline{Y}$ be the
fibres coming from $F_0$. Since the singular fibres of $\psi$ away
from $F_0$ form orbits under $u$, we have
$$e(F_0)=24-5r$$ where $r$ is the sum of the Euler numbers of singular
fibres of $Y$ away from $Y_0$.  If $r\le 1$, then $e(Y_0)\ge
e(Y)-1$, but no rational elliptic surface may have a fibre with
that big Euler number. Thus $r\ge 2.$

\medskip Case: $r=4$. Then $e(F_0)=24-5r=4$ and
$e(\overline{Y_0})=e(\overline{Y})-r=8$. Since $F_0$ is of type
$I_4$ or $IV$, we infer that
$$10=\rank~\NS(\overline{Y})=4+\rank~\NS(X/\langle u\rangle).$$ Then by
Proposition \ref{diminv}, $\dim H^2_{\rm
et}(Y,{\bbQ}_l)^u=\rank~\NS(X/\langle u\rangle)=6.$ Thus
$[u^*]=[1,\, 1.5,\, (\zeta_5:4).4]$. A monodromy argument excludes
the possibility $I_4$; if $F_0$ is of type $I_n$ or $I_m^*$, i.e.,
has a stable reduction of multiplicative type, then one can use
Tate's analytic uniformization to compute the representation of
the Galois group of the generic point of the strict local scheme
on the group of $l^t$-torsion points (the $l$-adic
representation), which leads to the monodromy computation. A fibre
of type $I_{5n}$ or $I_{5m}^*$ would be too big for a rational
elliptic surface if $n>1$ or $m>0$.

\medskip Case: $r=3$. Then $e(F_0)=24-5r=9$ and
$e(\overline{Y_0})=e(\overline{Y})-r=9$. Since $F_0$ is of type
$I_9$, $I_3^*$, or $III^*$, we infer that
$$10=\rank~\NS(\overline{Y})=\rank~\NS(X/\langle u\rangle).$$ Then by
Proposition \ref{diminv}, $\dim H^2_{\rm
et}(Y,{\bbQ}_l)^u=\rank~\NS(X/\langle u\rangle)=10.$ Thus
$[u^*]=[1,\, 1.9,\, (\zeta_5:4).3]$. A monodromy argument excludes
the possibilities $I_9$ and $I_3^*$.

\medskip Case: $r=2$. Then $e(F_0)=24-5r=14$ and
$e(\overline{Y_0})=e(\overline{Y})-r=10$. Thus $F_0$ is of type
$I_{14}$ or $I_{8}^*$, both are excluded by a monodromy argument.

\medskip Assume that $X^u$ consists of two points. Then by Theorem
1 \cite{DK1}, the quotient $X/\langle u\rangle$ has two rational
double points and the minimal resolution
$$Y\to X/\langle u\rangle$$ is a K3 surface. From a result of
Artin \cite{Artin} we see that each singularity of $X/\langle
u\rangle$ must be of type $E_8$, thus
$$\rank~\NS(Y)=16+\rank~\NS(X/\langle u\rangle)=16+\rank~\NS(X)^u.$$ Now by
  Proposition \ref{diminv}
$$22=\dim H^2_{\rm et}(Y,{\bbQ}_l)=\rank~\NS(Y) + \dim T_l^2(X)^u=16+H^2_{\rm et}(X,\bbQ_l)^u,$$
hence $H^2_{\rm et}(X,\bbQ_l)^u=6$. This gives the case (3).

\medskip Assume that $X^u$ consists of a point. Then by Theorem 1
\cite{DK1}, the minimal resolution $Y$ of $X/\langle u\rangle$ is
either a K3 surface or a rational surface.  Suppose that $X$
admits an $u$-invariant fibration of curves of arithmetic genus
$1$. Let $F_0$ be the fibre containing the point $X^u$. Since the
order 5 action of $u$ on $F_0$ has only one fixed point, we see
that $F_0$ is of type $II$, $III$ or $IV$. On the other hand,
$e(F_0)=24-5r$ which holds true even when $u$ acts on the base
trivially, where $r$ is the sum of the Euler numbers of singular
fibres of $Y$ away from the singular fibre coming from $F_0$. Thus
$F_0$ is of type $IV$ and $r=4$. If $X/\langle u\rangle$ is not rational, then the minimal resolution $Y$ is a K3 surface and the fibre $Y_0$ corresponding to $F_0$ consists of
the proper image of the 3 components of $F_0$ and the 8 components lying over the $E_8$-singularity, hence $e(Y_0)<20=e(Y)-4$, a contradiction. Thus $X/\langle u\rangle$ is rational,
yielding the case (5). Suppose that $X$ does not admit an
$u$-invariant elliptic fibration. Then by Proposition 2.9
\cite{DK2}, $Y$ cannot be a rational surface. Thus $Y$ is a K3
surface and we have the case (4).
\end{proof}

\begin{remark} The case (3) is also supported by examples. In
fact, there are 2 dimensional family of elliptic K3 surfaces with
a 5-torsion and with 2 fibres of type $II$ and 4 fibres type $I_5$
(\cite{ItoL} Theorem 4.4). The automorphism induced by a 5-torsion
has 2 fixed points, the cusps of the two type $II$ fibres.
\end{remark}

The following is the main result of this section.

\begin{theorem} \label{5main} Let $g$ be an automorphism of finite order of a K3 surface $X$ over an algebraically closed
field of characteristic $p=5$. Assume that the order of $g$ is
divisible by $5$. Then $${\rm ord}(g) = 5n, \,\,\,{\rm
where}\,\,\,n=1,\,2,\,3,\,4,\,6, \, 8.$$ There are examples supporting
these orders $($see Examples \ref{exam5} and \ref{exam5-2}$)$.
\end{theorem}

The proof follows from the following lemmas \ref{kod2}, \ref{25} -- \ref{1pt}.

\begin{lemma}\label{kod2} Let $g$ be an automorphism of order
$5n$. Let $u=g^n$. If the fixed locus $X^u$ contains a curve of
arithmetic genus $2$, then $n\le 8$, $n\neq 5,\, 7$.
\end{lemma}

\begin{proof} Assume that $X^u$ contains a curve $C$ of arithmetic genus 2. Then by \cite{DK2} Proposition
2.3, the linear system $|C|$ defines a double cover $X\to {\bf P}^2$
and $X$ is birationally isomorphic to the surface
 \begin{equation}\label{char5}
z^2=(y^5-yx^4)P_1(x_0,x)+ P_6(x_0,x)
 \end{equation}
where $(x_0: x: y)$ is the homogeneous coordinates of ${\bf P}^2$,
$P_i(x_0,x)$ is a homogeneous polynomial of degree $i$, and the
induced automorphism $\overline{u}$ of ${\bf P}^2$ is given by
$$\overline{u}(x_0,x,y)=(x_0,x,x+y).$$ Since $g$
preserves $|C|$, it induces a linear automorphism $\overline{g}$
of ${\bf P}^2$. It follows that $n$ is not a multiple of $5,$ since ${\bf P}^2$  in characteristic $p$ cannot
admit an automorphism of order $p^2$. Note
that $g$ is induced from $\overline{g}$, possibly composed with
the covering involution. Since $\overline{g}^n=\overline{u}$, either $\overline{g}$ has order $5n$ or $\frac{5}{2}n$. The latter happens if and only if
$n$ is even and $g^{\frac{5n}{2}}$ is the covering involution. Since $g$ preserves the equation
(\ref{char5}), $\overline{g}$ preserves the right hand side of (\ref{char5}) up to a scalar multiple.

\medskip
Suppose that $\overline{g}$ has order $5n$.
Using Jordan canonical form we may assume that
$$\overline{g}^n=\overline{u}=\begin{pmatrix}
1&0&0\\0&1&0\\0& 1& 1
\end{pmatrix},\quad\overline{g}=\begin{pmatrix}
\zeta_n&0&0\\0&1&0\\0& \frac{1}{n}& 1
\end{pmatrix}$$
where $\zeta_n$ is a primitive $n$-th root of 1.
If $n\ge 7$, then it is easy to see that there are two possibilities
for the pair $(P_1,\,\,P_6)$
$$(P_1,\,\,P_6)=(Ax,\,\,Bx^6)\,\,\,{\rm or}\,\,\,(Ax_0,\,\,Bx_0x^5),$$ and in either case the
point $(x_0:x:y,z)=(1:0:0,0)$ is a non-rational double point, hence (\ref{char5}) does not
define a K3 surface. Thus $n\le 6$, $n\neq 5$.\\
 If $n=6$, then there are two possibilities
$$(P_1,\,\,P_6)=(Ax,\,\,B_0x_0^6+B_6x^6)\,\,\,{\rm or}\,\,\,(Ax_0,\,\,Bx_0x^5).$$ The second case does not define a K3 surface. Neither does the first case with $AB_0=0$. Assume the first case with $AB_0\neq 0$. Then  the equation $$z^2=Ax(y^5-yx^4)+ B_0x_0^6+B_6x^6$$ defines a (smooth)  K3 surface and admits an automorphism of order $30$
$$g(x_0:x:y,z)=(\zeta_6x_0:x:x+y, \pm z).$$ 
Moreover, by rescaling coordinates we may assume that $A=B_0=1$.

\medskip
Suppose that $n=2m$ and $\overline{g}$ has order $5m$. Since $\overline{g}^{2m}=\overline{u}=\overline{u}^6$, $\overline{g}^m=\overline{u}^3$ and we may assume that
$$\overline{g}=\begin{pmatrix}
\zeta_m&0&0\\0&1&0\\0& \frac{3}{m}& 1
\end{pmatrix}.$$
By a similar computation as above, $m\le 6$, $m\neq 5$. If $m=6$, then, as we saw in the above,
the order of $g$ would be $5m=30$. Thus $m\le 4$.
\end{proof}

\begin{example}\label{exam5} Take $m=4$. Then there are two possibilities for the pair
$$(P_1,\,\,P_6)=(Ax,\,\,B_2x_0^4x^2+B_6x^6)\,\,\,{\rm or}\,\,\,(Ax_0,\,\, B_1x_0^5x+B_5x_0x^5).$$
In the first case the surface has a non-rational double point, so does not yield a K3 surface. In the second, $A\neq 0$, so we may assume $A=1$. The surface 
$$z^2=x_0(y^5-yx^4+B_1x_0^4x+B_5x^5)$$
has at worst rational double points iff  $B_1\neq 0$ iff the quintic is smooth. Assume $B_1\neq 0$. Then the equation has 5 rational double points of type $A_1$, $g$ must map $z^2$ to $\zeta_4z^2$, hence by putting $\zeta_4=\zeta_8^2$ we can rewrite $g$ as
$$g(x_0:x:y,z)=(\zeta_8^2x_0:x:2x+y, \zeta_8z)$$
which is of order 40. By rescaling, we may assume $B_1=1$. This is the only example of order 40, as will be shown in this section.
\end{example}

\begin{lemma}\label{25} There is no automorphism of order $5^2$.
\end{lemma}

\begin{proof} Suppose that ${\rm ord}(g)=25$. Then  $g^*|H^2_{\rm
  et}(X,\bbQ_l)$, $l\neq 5$, has eigenvalues
$$[g^*]=[1,\, 1,\, \zeta_{25}:20].$$
Thus $u=g^5$ has $[u^*]=[1,\, 1,\, (\zeta_5:4).5]$, which does not
occur by Proposition \ref{5} and Lemma \ref{kod2}.
\end{proof}

\begin{lemma}\label{kod1} Let $g$ be an automorphism of order
$5n$, $5\nmid n$.  If $\kappa(X,X^u)=1$ for $u=g^n$, then
$n=1,2,3,4, 6$.
\end{lemma}

\begin{proof} Since $\kappa(X,X^u)=1$, we are in the case (2) of Proposition \ref{5}. The fixed locus $X^u$  is the
support of a fibre $F_0$ of a fibration of curves of arithmetic
genus $1$, and $F_0$ is of type $IV$ or $III^*$. Note that $g$
acts on $X^u$.

\medskip Case 1. $F_0$ is of type $IV$. Then $u=g^n$ has
$$[u^*]=[1,\, 1.5,\, (\zeta_5:4).4].$$

\medskip\noindent
Claim:  $n\neq 3^2$, $n$ cannot be a prime $\ge 7$.\\
Suppose that $n=3^2$. Then by Theorem \ref{faith}, either
$\zeta_9$ or $\zeta_{45}$ must appear in $[g^*]$. But $\phi(9)>5$
and $\phi(45)>16$. Similarly, $n$ cannot be a prime $\ge 7$.

\medskip\noindent
Claim:  $n\neq 2^3$.\\
Suppose that $n=2^3$. Then $g$ preserves one component of $X^u$
and interchanges or preserves  the other two components. Thus the
action of $g$ on the 3 components of $X^u$ has eigenvalues
$[1,\,1,\,\pm 1]$ and we infer that
$$[g^*]=[1,\,1,\,1,\,\pm 1,\, \eta_1,\,\eta_2,\, \zeta_{40}:16]$$
where $\eta_1,\,\eta_2$ is a combination of $\zeta_4:2$, $-1$,
$1$. The involution $s=g^{20}$ preserves each of the three
components of $F_0$. It is tame and has $e(X^s)\neq 8$, hence is
non-symplectic. Thus there is a curve $C$ in ${\rm Fix}(s)$
passing through the singular point $P$ of $F_0$. The curve $C$ may be
equal to one of the 3 components of $F_0$. In any case we get a
contradiction since a tame involution, locally at $P\in C$,
can preserve at most one curve other than $C$.

\medskip\noindent
Claim:  $n\neq 12$.\\
Suppose that $n=12$. Then the tame involution $s=g^{30}$ preserves
each of the three components of $F_0$. Assume that $s=g^{30}$ is
non-symplectic. Then we get a contradiction as in
the previous case. Assume that $s=g^{30}$ is symplectic. Then $s$
has 8 fixed points on $X$ by Proposition \ref{sym}. We infer that
it has 4 fixed points on $F_0$. Since $g(F_0)=F_0$, $g$ acts on
the base of the elliptic fibration. Since $u=g^{12}$ acts
non-trivially on the base, by Lemma \ref{P1} $g^5$ acts trivially.
Thus $s=g^{30}$ acts on each singular fibre. There are at least 5
singular fibres, a $\langle g\rangle$-orbit,  away from $F_0$ and on each of them $s$ has a
fixed point.

\medskip Case 2. $F_0$ is of type $III^*$. Then $u=g^n$ has
$$[u^*]=[1,\, 1.9,\, (\zeta_5:4).3].$$

\medskip\noindent
Claim:  $n\neq 3^2$, $n\neq$ a prime $\ge 7$.\\
Suppose that $n=3^2$. The order 9 action of $g$ on $X^u$ preserves
each of the 8 components. Thus $[1,\,1.8]\subset [g^*]$, hence
$\zeta_9\notin [g^*]$ and by Theorem \ref{faith} $\zeta_{45}\in
[g^*]$. But $\phi(45)>12$. Similarly, $n$ cannot be a prime $\ge
7$.

\medskip\noindent
Claim:  $n\neq 2^3$.\\
Suppose that $n=2^3$. Then $g$ acts on $X^u$, flipping or
preserving the configuration of the 8 components. Thus the action
of $g$ on the 8 components of $X^u$ has eigenvalues $[1.8]$ or
$[1.5,\,-1.3]$, hence $\zeta_8\notin [g^*]$. Then by Theorem
\ref{faith}, $\zeta_{40}\in [g^*]$. But $\phi(40)>12$.

\medskip\noindent
Claim:  $n\neq 12$.\\
Suppose that $n=12$. As in the previous case, the action of $g$ on
the 8 components of $X^u$ has eigenvalues $[1.8]$ or
$[1.5,\,-1.3]$, hence $[g^*]$ contains either $[1,\,1.8,\,\pm 1]$
or $[1,\,1.5,\,-1.3,\,\pm 1]$. Since $\phi(60)>12$, we infer that
$[(\zeta_5:4).3]$ in $[u^*]$ must come from a combination of
$\zeta_{30}:8$, $\zeta_{20}:8$, $\zeta_{15}:8$, $\zeta_{10}:4$,
$\zeta_5:4$ in $[g^*]$. In any case, $g^{30}=1$ or $g^{20}=1$.
\end{proof}

\begin{lemma}\label{2pts} Let $g$ be an automorphism of order
$5n$, $5\nmid n$. If $g^n$ fixes at most two points and has $[g^{n*}]=[1,\, 1.5,\, (\zeta_5:4).4]$  where the first eigenvalue corresponds to a g-invariant ample divisor, then
$n=1,2,3,4, 6$.
\end{lemma}

\begin{proof} Let
$u:=g^{n}.$ Then $u$ has order 5. 

\medskip\noindent
Claim:  $n\neq 3^2$, $n$ cannot be a prime $\ge 7$.\\
The proof is the same as in Case 1, Lemma \ref{kod1}.

\medskip\noindent
Claim:  $n\neq 2^3$.\\
Suppose that $n=2^3$. Then $u=g^8$. Since $g$ acts on the set ${\rm Fix}(u)$ and ${\rm Fix}(g^2)\subset {\rm
Fix}(g^4)\subset {\rm Fix}(g^{8})$, we see that
$${\rm Fix}(g^2)={\rm Fix}(g^4)={\rm Fix}(g^{8})= \{{\rm two\,\,points}\}\,\,{\rm or}\,\,\{{\rm one\,\,point}\}.$$
The proof will be divided into 3 cases according to the
possibility of $[g^*]$:

\medskip
(1) $[1,\,\zeta_8:4,\,\pm 1,\, \zeta_{40}:16]$,

(2) $[1,\,\eta_1,\ldots, \eta_5,\, \zeta_{40}:16]$, where $\eta_i$'s are combination of $\zeta_4:2$, $\pm 1$,

(3) $[1,\,\zeta_8:4,\,\pm 1,\, \tau_1,\ldots,\tau_{16}]$, 
where $\tau_i$'s are combination of $\zeta_{20}:8$, $\pm\zeta_5:4$.

\medskip\noindent
Case (1). 
The tame involution $s=g^{20}$ has $$[g^{20*}]=[1,\,-1.4,\,1,\, -1.16],\,\,\,e(g^{20})=-16,$$ hence is non-symplectic. Thus  ${\rm Fix}(g^{20})$
consists of either a curve $C_9$ of genus 9 or a smooth rational
curve $R$ and a curve $C_{10}$ of genus 10. The tame automorphism
$g^{10}$ has $e(g^{10})=4$. Since ${\rm
Fix}(g^{10})\subset {\rm Fix}(g^{20})$, we infer that ${\rm Fix}(g^{10})$
consists of either 4 points or $R$ and 2 points on
$C_{10}$. In any case, $g$ acts on ${\rm Fix}(g^{10})$, hence
$g^2$ fixes at least 3 points, a contradiction.

\medskip\noindent
Case (2). In this case, $s=g^{20}$ has $[g^{20*}]=[1,\,1.5,\, -1.16]$, $e(g^{20})=-8$, hence is non-symplectic. Thus ${\rm Fix}(g^{20})$ consists
of a curve $C_{d+5}$ of genus $d+5$ and $d$ smooth rational
curves, where $0\le d\le 5$. Applying Deligne-Lusztig (Proposition
\ref{DL}) to $g^{28}=su$, we get
$$\Tr (u^*|H^*(X^{s}))=\Tr (g^{28*}|H^*(X))=12.$$ If $d\le 2$, then $$\Tr
(u^*|H^*(X^{s}))= 2d+\Tr (u^*|H^*(C_{d+5}))\le 7.$$ Thus $d\ge 3$. If $d=3$ or 4, then $g^8$ preserves each of the $d$
smooth rational curves, hence  fixes at least $d$ points, a
contradiction. If $d=5$, then either the 5 smooth rational curves form a single orbit under $u=g^8$ or each of them is preserved by $g^8$. In the former case  $$\Tr
(u^*|H^*(X^{s}))= \Tr (u^*|H^*(C_{10}))\le 7.$$ In the latter $g^8$ fixes at least one point on each of the 5 smooth rational curves, a contradiction.

\medskip\noindent
Case (3). In this case,
$g^{20}$ has $[g^{20*}]=[1,\,-1.4,\,1,\,\, 1.16]$, $e(g^{20})=16$, hence is non-symplectic. Thus $X^s$ is
either a union of  $8$ smooth rational curves and possibly some
elliptic curves or a union of a curve $C_{d-7}$ of genus $d-7$ and
$d$ smooth rational curves, where $9\le d\le 17$. In the first
case, the order 5 action of $u=g^{8}$ on ${\rm Fix}(g^{20})$
preserves at least 3 among the 8 smooth rational curves, hence
fixes at least 3 points, a contradiction. Thus we are in the second
case. Deligne-Lusztig (Proposition \ref{DL}) does not work here,
so we need a different argument.  The action of $g^4$ on $X^s$ is
of order 5 with at most 2 fixed points. This is possible only if
$d=5r$, $5r+1$ or $5r+2$, and the $5r$ smooth rational curves form $r$ orbits under  $g^4$.
Since $9\le d\le 17$, $r=2$ or 3 and each orbit (and $C_{d-7}$ also) gives a $g^{4}$-invariant
divisor, hence $\rank~\NS(X)^{g^{4}}\ge 3$. But $\dim H^2_{\rm
et}(X,{\bbQ}_l)^{g^{4}}=2$.

\medskip
Claim:  $n\neq 12$.\\
Suppose that $n=12$. Then $u=g^{12}$. Since $g$ acts on the set ${\rm Fix}(u)$ and ${\rm Fix}(g^i)\subset {\rm
Fix}(g^{12})$ for any $i$ dividing $12$, we see that
$${\rm Fix}(g^2)={\rm Fix}(g^4)={\rm Fix}(g^{6})= {\rm Fix}(g^{12})=\{{\rm two\,\,points}\}\,\,{\rm or}\,\,\{{\rm one\,\,point}\}.$$
First we claim that neither $g^{30}$ nor $g^{20}$ is symplectic. Suppose that $g^{30}$ is symplectic. Then  by
Proposition \ref{sym}, ${\rm Fix}(g^{30})$ consists of 8 points,
on which $g$ acts. This permutation has order dividing 30, so in its cycle decomposition each cycle has length 1, 2, 3, 5, 6, hence  $g^6$
must fix at least 3 points. This proves that $g^{30}$ is not
symplectic. Suppose that $g^{20}$ is symplectic. Then ${\rm
Fix}(g^{20})$ consists of 6 points, on which $g$ acts with order
dividing 20. Its cycle decomposition must contain a cycle of length 5, because otherwise $g^4$ would fix all the 6 points. So $g^5$ fixes all the 6 points and 
${\rm Fix}(g^5)={\rm Fix}(g^{10})= {\rm Fix}(g^{20})=\{{\rm 6\,\,points}\}.$ Note that $[g^{20*}]=[1,\, 1.9,\, (\zeta_3:2).6]$. Thus $[\zeta_{60}:16]$ cannot appear in $[g^*]$.  It is easy to check that neither $[\zeta_{20}:8]$ nor $[\zeta_{12}:4]$ can appear in $[g^*]$.  Indeed, if $[\zeta_{20}:8]\subset [g^*]$, then $e(g^{10})<6$, and if $\zeta_{20}\notin [g^*]$ and $[\zeta_{12}:4]\subset [g^*]$, then $e(g^{10})>6$. This implies that $g^{30}=1$. This proves that $g^{20}$ is
not symplectic. The claim is proved, hence ${\rm ord}(g)=5.12$. Now $\zeta_{12}\in
[g^{5*}]$, thus  $\zeta_{12}$ or $\zeta_{60}\in [g^*]$. We have 4 cases for$[g^*]$: 

\medskip
(1) $[1,\,\zeta_{12}:4,\,\pm 1,\, \zeta_{60}:16]$ or $[1,\,\zeta_{12}:4,\,\pm 1,\, (\zeta_{20}:8).2]$,

(2)  $[1,\,\zeta_{12}:4,\,\pm 1,\, \zeta_{20}:8,\,\eta_1,\ldots, \eta_8]$, 

(3) $[1,\,\zeta_{12}:4,\,\pm 1,\,\eta_1,\ldots, \eta_{16}]$,

(4) $[1,\,\tau_1,\ldots, \tau_{5},\,\zeta_{60}:16]$ 

where $\eta_i$'s are combination of $\pm\zeta_{15}:8$, $\pm\zeta_{5}:4$, and $\tau_i$'s combination of 

$\pm\zeta_{3}:2$, $\zeta_{4}:2$, $\pm 1$.

\medskip\noindent
Case (1). In this case, $[g^{30*}]=[1,\,-1.4,\,1,\, -1.16]$, $e(g^{30})=-16.$ Thus  ${\rm Fix}(g^{30})$
consists of either a curve $C_9$ of genus 9 or a smooth rational
curve $R$ and a curve $C_{10}$ of genus 10. The tame automorphism
$g^{10}$ has $e(g^{10})=14$ or $-10$. Since ${\rm
Fix}(g^{10})\subset {\rm Fix}(g^{30})$,  ${\rm Fix}(g^{10})$
consists of either 14 points or the curve $R$ and 12 points on
$C_{10}$. In any case, the order 5 action $g^2$ on ${\rm Fix}(g^{10})$ fixes at least 3 points, a contradiction.

\medskip\noindent
Case (2). In this case, $[g^{30*}]=[1,\,-1.4,\,1,\, -1.8,\,1.8]$, $e(g^{30})=0.$ The tame automorphism
$g^{10}$ has $e({\rm Fix}(g^{10}))=6$ or $-6$.
The fixed locus ${\rm Fix}(g^{30})$
consists of either elliptic curves or a curve $C_{d+1}$ of genus $d+1$ and $d$ smooth rational
curves, $1\le d\le 9$. In the second case, since $g^6$ has $\dim H^2_{\rm
et}(X,{\bbQ}_l)^{g^{6}}=2$ and preserves $C_{d+1}$, it has at most one orbit on the $d$ smooth rational
curves. Each orbit has length 1 or 5, so $d=1$ or 5.   If $d=5$, then $g^{10}$ preserves each of the 5 smooth rational
curves, hence $e(g^{10})\ge 10$ if $g^{10}|C_6$ is not the identity and  $e(g^{10})=0$ if the identity. If $d=1$, then, since $e(g^{10})=\pm 6$, $g^{10}$ fixes 4 points on $C_2$, then $g^2$ fixes the 4 points. If ${\rm Fix}(g^{30})$
consists of elliptic curves, then  $g^{10}$ fixes 3 points on one of the elliptic curves, 3 points on another and some of the elliptic curves, then $g^2$ fixes the 6 points.

\medskip\noindent
Case (3). In this case, $[g^{30*}]=[1,\,-1.4,\,1,\,\,1.16]$, $e(g^{30})=16.$ Thus ${\rm Fix}(g^{30})$
consists of either 8 smooth rational
curves and possibly some elliptic curves or a curve $C_{d-7}$ of genus $d-7$ and $d$ smooth rational
curves, $9\le d\le 17$. In the second case, since $g^6$ has $\dim H^2_{\rm
et}(X,{\bbQ}_l)^{g^{6}}=2$ and preserves $C_{d-7}$, it has at most one orbit on the $d$ smooth rational
curves. This is impossible, since each orbit has length 1 or 5.  In the first case the action of $g^6$ on the 8 rational curves must preserve at least 3 of them, hence leads to a contradiction.

\medskip\noindent
Case (4).
This case can be ruled out by the same argument as above.
\end{proof}

\begin{lemma}\label{1ptK3} Let $u$ be an automorphism of order $5$. If $u$ fixes exactly one point and  $X/\langle
u\rangle$ is a K3 surface with one rational double point. Then $\rank\, {\rm NS}(X)^u\le 4$.
\end{lemma}

\begin{proof} There is an $u$-invariant ample divisor class, so the lattice ${\rm NS}(X)^u$ is hyperbolic. Suppose that
$\rank~{\rm NS}(X)^u\ge 5.$ Since a hyperbolic lattice of rank $\ge 5$ represents 0, so does ${\rm NS}(X)^u$.  Thus there is an $u$-invariant effective divisor class $E$ with $E^2=0$ of the form
$$E=E'+R_1+\ldots+R_t$$ where $E'$ is an elliptic curve,
$R_i$ a string of smooth rational curves, possibly non-reduced, such that the supports of $R_i$'s are mutually disjoint. Since $u$ preserves the class $E$, it preserves the class of $E'$ and permutes $R_i$'s. A composition of Picard-Lefschetz reflections sends $E$ to $E'$. The linear system $|E'|$ is an
$u$-invariant elliptic fibration, contradicting Proposition
\ref{5}.
\end{proof}

\begin{lemma}\label{1pt} Let $g$ be an automorphism of order
$5n$, $5\nmid n$. If $g^n$ fixes exactly one point and has $[g^{n*}]=[1,\, 1.13,\, (\zeta_5:4).2]$  where the first eigenvalue corresponds to a g-invariant ample divisor, then
$n=1,2,3,4, 6$. 
\end{lemma}

\begin{proof} By Proposition \ref{5}, this is the case where $X/\langle
g^n\rangle$ is a K3 surface with one rational double point  and $X$ does not admit  a $g^n$-invariant elliptic fibration. The latter implies that ${\rm Fix}(g^{5i})$, $i\ge 1$, contains no elliptic curve, as $g^n$ acts on ${\rm Fix}(g^{5i})$, hence preserves elliptic curves in ${\rm Fix}(g^{5i})$ if any.

\medskip
Claim:  $n\neq 2^3$.\\
Suppose that $n=2^3$. Then $u=g^8$ and
$${\rm Fix}(g)={\rm Fix}(g^2)={\rm Fix}(g^4)={\rm Fix}(g^{8})= \{\rm a\,\,point\}.$$
If $g^{20}$ is symplectic, then
${\rm Fix}(g^{20})$ consists of 8 points on which $g^4$ acts. By considering orbit decomposition of the 8 points under $g^4$, we infer
that $g^4$ fixes three or all of them, a contradiction. Thus  $g^{20}$
is non-symplectic and the quotient $X/\langle  g^{20}\rangle$ is
rational. Write
$$[g^{*}]=[1,\,(\zeta_{8}:4).a,\,\,\eta_1,\ldots,\eta_{13-4a},\,\tau_1,\ldots, \tau_{8}]$$ where $\eta_i$'s are combination of $\zeta_{4}:2$ and $\pm 1$, and  $\tau_i$'s combination of $\pm\zeta_{5}:4$, $\zeta_{20}:8$. Since $X/\langle  g^{20}\rangle$ is
rational, by Proposition \ref{diminv} $$\rank~{\rm
NS}(X)^{g^{20}}=\dim H^2_{\rm et}(X,{\bbQ}_l)^{g^{20}}.$$ It follows that
$\eta_1,\ldots,\eta_{13-4a}$ are supported on ${\rm NS}(X)$ and
$$\rank~{\rm
NS}(X)^{g^8}\ge 1+(13-4a).$$ Then by Lemma \ref{1ptK3}, $14-4a\le 4$, i.e. $a=3$ and $$[g^{*}]=[1,\,(\zeta_{8}:4).3,\,\,\pm 1,\,\tau_1,\ldots, \tau_{8}]$$ where $\tau_i$'s are combination of $\pm\zeta_{5}:4$, $\zeta_{20}:8$.
Then we compute $$[g^{20*}]=[1,\,-1.12,\,1,\,\,1.8],\quad e(g^{20})=0,\quad e(g^{10})=-4\,\,{\rm or}\,\, 12.$$
Thus ${\rm Fix}(g^{20})$ is not empty. Since ${\rm Fix}(g^{20})$ contains no elliptic curve, it consists of a
curve $C_{d+1}$ of genus $d+1$ and $d$ smooth rational curves,
where $1\le d\le 9$. Note that $\dim H^2_{\rm
et}(X,{\bbQ}_l)^{g^{4}}=2$. Thus the action of $g^4$ on the $d$ smooth rational curves has at most 1 orbit. This is possible only if $d=1$ of $5$. 
In any case, the tame automorphism $g^{10}$ preserves
each of the $d$ smooth rational curves. 
Since $e(g^{10})=-4$
or $12$, we conclude that $g^{10}$
fixes 10 points on $C_{2}$ if $d=1$, which is impossible as no involution of a genus 2 curve can fix 10 points, and fixes 2 points on  $C_{6}$ if $d=5$, each of which must be fixed by $g^2$.

\medskip
Claim:  $n\neq 3^2$.\\
Suppose that $n=3^2$. Then $u=g^9$ and
$${\rm Fix}(g)={\rm Fix}(g^3)={\rm Fix}(g^9)= \,\,\{{\rm a\,\,point}\}.$$ By Theorem \ref{faith},
$\zeta_9\in [g^*]$. Assume $[g^*]=[1,\,\zeta_9:6,\,\eta_1,\ldots,
\eta_{7},\,\tau_1,\ldots, \tau_8]$ where $\eta_i$'s are combination of $\zeta_{3}:2$, $1$, and $\tau_i$'s combination of $\zeta_{15}:8$,
$\zeta_5:4$. Then the order 3 tame automorphism $g^{15}$ has \\
 $[g^{15*}]=[1,\,(\zeta_3:2).3,\,\, 1.7,\,\,1.8]$, $e(g^{15})=15$, hence is non-symplectic
by Lemma \ref{Lefschetz} and the quotient $X/\langle g^{15}\rangle$ is
rational. Then by Proposition \ref{diminv}, $\rank~{\rm
NS}(X)^{g^{15}}=\dim H^2_{\rm et}(X,{\bbQ}_l)^{g^{15}}$. It follows that
$\eta_1,\ldots, \eta_{7}$ are supported on ${\rm NS}(X)$, hence $\rank~{\rm
NS}(X)^u\ge 8,$  contradicting Lemma
\ref{1ptK3}. This proves that
$$[g^*]=[1,\,(\zeta_9:6).2,\,1,\,\tau_1,\ldots, \tau_8]$$ where
$[\tau_1,\ldots, \tau_8]=[\zeta_{15}:8]$ or $[(\zeta_5:4).2]$. Then $$[g^{15*}]=[1,\,(\zeta_3:2).6,\,\, 1,\,\,1.8], \quad e(g^{15})=6,\quad \Tr (g^{24*}|H^*(X))=-4.$$ 
The order 5 action of $u=g^9$ on ${\rm Fix}(g^{15})$ has exactly one fixed
point.

Assume that $u=g^9$ fixes an isolated point of ${\rm Fix}(g^{15})$. Then ${\rm Fix}(g^{15})$
consists of $5t+1$ points, $5d$ smooth rational curves, and possibly a curve $C_r$ of genus $r\ge 2$. (Note that ${\rm Fix}(g^{5i})$ contains no elliptic curve).
Since $u$ acts freely on $C_r$, it has trace $\Tr (u^*|H^*(C_r))=0$, so
$$\Tr (u^*|H^*(X^{g^{15}}))=\Tr (u^*|H^*({\rm a\,\,point}))+\Tr (u^*|H^*(C_r))= 1$$ and $\Tr (u^*|H^*(X^{g^{15}}))\neq \Tr (g^{24*}|H^*(X))$, which contradicts Deligne-Lusztig.

Assume that $u=g^9$ fixes a point of a smooth rational curve in
${\rm Fix}(g^{15})$. Then ${\rm Fix}(g^{15})$ consists of $5t$ points, $5d+1$ smooth rational
curves,  and possibly a curve $C_r$ of genus $r\ge 2$ on which $u$ acts feely. Then $$\Tr (u^*|H^*(X^{g^{15}}))=\Tr (u^*|H^*({\bf P}^1))+\Tr (u^*|H^*(C_r))= 2,$$
 again $\Tr (u^*|H^*(X^{g^{15}}))\neq \Tr (g^{24*}|H^*(X))$.
 
Assume that $u=g^9$ fixes a point of a curve $C_r$ of genus $r\ge 2$ in
${\rm Fix}(g^{15})$. Then ${\rm Fix}(g^{15})$ consists of $5t$ points, $5d$ smooth
rational curves and the curve $C_r$.  Since $e(g^{15})=5t+10d+2-2r=6$, we see that $t=2t'$ for some $t'$ and
$r=5t'+5d-2\ge 3$. Deligne-Lusztig does not work here, as $u^*|H^1_{\rm et}(C_r,{\bbQ}_l)$ may have
$$\Tr [u^*|H^1_{\rm et}(C_r,{\bbQ}_l)]=\Tr [1.(2t'+2d+4),\, (\zeta_5:4).(2t'+2d-2)]=6.$$ The action of
$g$ on the $10t'$ isolated points of ${\rm Fix}(g^{15})$ has orbits of length 15
or 5, since $g^9$ has orbits of length 5 only. We see that
$g^5$ fixes no point in any orbit of $g$ of length 15, and fixes
each point in any orbit of length 5. Since $\dim H^2_{\rm et}(X,{\bbQ}_l)^{g^{15}}=10$, we see that $d\le 1$ and
$g^5$ preserves each of the $5d$ smooth rational curves. If
$C_r\subset {\rm Fix}(g^{5})$, then
$$e(g^{5})=e(g^{15})-15b=6-15b$$ for some $b\ge 0$, which contradicts the direct computation of trace $$\Tr
(g^{5*}|H^*(X))=0\,\,{\rm or}\,\, 12.$$ Thus
$C_r\nsubseteq {\rm Fix}(g^{5})$, i.e. $g^5$ acts non-trivially on $C_r$. Since $g$
has a fixed point on $C_r$, so does $g^5$. Thus $e(g^{5})>0$, hence $e(g^{5})=12$. From this we
infer that $g^5$ has 12, 7 or 2 fixed points on $C_r$. In any
case, $g$ acts on these points, hence fixes at least 2 points.

\medskip
Claim:  $n$ cannot be a prime $\ge 17$.\\
Suppose that $n$ is a prime $\ge 17$. Then by Theorem \ref{faith},
either $\zeta_n$ or $\zeta_{5n}$ must appear in $[g^*]$. But
$\phi(n)>13$ and $\phi(45)>8$. 

\medskip
Claim:  $n\neq 11$.\\
Suppose that $n=11$. Then by Theorem \ref{faith}, $\zeta_{11}$
must appear in $[g^*]$ and
$$[g^*]=[1,\,\zeta_{11}:10,\,1.3,\,(\zeta_5:4).2].$$  The automorphism $s=g^5$ is non-symplectic of order 11 and has
$$[s^*]=[g^{5*}]=[1,\,\zeta_{11}:10,\,1.3,\,\,1.8], \quad e(g^{5})=13,\quad \Tr
(g^{16*}|H^*(X))=3.$$  The order 5 action of $g$ and $u=g^{11}$ on ${\rm Fix}(g^5)$ has
exactly one fixed point.

Assume that $u=g^{11}$ fixes an isolated point of ${\rm Fix}(g^5)$. Then
${\rm Fix}(g^5)$ consists of $5t+1$ points, $5d$ smooth rational curves, and possibly a curve
$C_r$ of genus $r\ge 2$  on which $u$ acts feely. Then $\Tr (u^*|H^*(C_r))=0$, so $$\Tr (u^*|H^*(X^{g^{5}}))=\Tr (u^*|H^*({\rm a\,\,point}))+\Tr (u^*|H^*(C_r))= 1$$ and $\Tr (u^*|H^*(X^{g^{5}}))\neq \Tr (g^{16*}|H^*(X))$, which contradicts Deligne-Lusztig.

Assume that $u=g^{11}$ fixes a point of a smooth rational curve in
${\rm Fix}(g^5)$. Then ${\rm Fix}(g^5)$ consists of $5t$ points, $5d+1$ smooth rational
curves, and possibly a curve
$C_r$ of genus $r\ge 2$  on which $u$ acts feely. Then $$\Tr (u^*|H^*(X^{g^{5}}))=\Tr (u^*|H^*({\bf P}^1))+\Tr (u^*|H^*(C_r))= 2,$$
 again $\Tr (u^*|H^*(X^{g^{5}}))\neq \Tr (g^{16*}|H^*(X))$.
 
 Assume that $u=g^{11}$ fixes a point of a curve $C_r$ of genus $r\ge 2$
in ${\rm Fix}(g^5)$. Then ${\rm Fix}(g^5)$ consists of $5t$ points, $5d$ smooth
rational curves $R_1,\ldots, R_{5d}$ and the curve $C_r$.  Since
$e(g^5)=5t+10d+2-2r=13$,
$2r-2=5t+10d-13\ge 2.$ Deligne-Lusztig does not work
here, as  $u^*|H^1_{\rm et}(C_r,{\bbQ}_l)$ may have
$$\Tr [u^*|H^1_{\rm et}(C_r,{\bbQ}_l)]=\Tr [1.(t+2d-3),\,\, (\zeta_5:4).(t+2d-2)]=-1.$$ Since $g^5$ is
non-symplectic of order 11, we may assume that
$g^{5*}\omega_X=\zeta_{11}^5\omega_X$. Then an isolated fixed point of
$g^5$ is one of the following 5 types:
$\frac{1}{11}(1,4)$, $\frac{1}{11}(2,3),\,\,\frac{1}{11}(6,10),\,\,\frac{1}{11}(7,9),\,\,\frac{1}{11}(8,8).$
Let $5t_i$  be the number of isolated fixed point of the $i$-th
type. Then $$\sum t_i=t.$$ The quotient $X'=X/\langle g^5\rangle$
is a rational surface with
$$K_{X'}=-\frac{10}{11}\Big(C_r'+\sum_{i=1}^{5d} R_i'\Big)$$
where $C_r'$ and $R_i'$ are the images of $C_r$ and $R_i$. Let
$\varepsilon: Y\to X'$ be a minimal resolution. Then
$K_Y=\varepsilon^*K_{X'}-\sum D_p$
where $D_p$ is an effective $\bbQ$-divisor supported on the
exceptional set of the singular point $p\in X'$. Thus
$$K_Y^2=K_{X'}^2+\sum D_p^2=K_{X'}^2-\sum K_YD_p.$$
See, e.g., Lemma 3.6 \cite{HK1} for the formulas of $D_p$ and
$K_YD_p$, which are valid for tame quotient singular points in
positive characteristic.
We compute 
$$K_Y^2= 10-\rho(Y)=
10-\{\rho(X')+10t_1+20t_2+25t_3+10t_4+5t_5\}$$ 
$$=-2-10t_1-20t_2-25t_3-10t_4-5t_5$$
where we use $\rho(X')=\dim H^2_{\rm
et}(X,{\bbQ}_l)^{g^5}=12$.\\
Since
$C_r'^2=11(5t+10d-13)$ and $R_i'^2=-22$, 
we have
$$K_{X'}^2=\frac{10^2}{11^2}\{ 11(5t+10d-13)-22\cdot
5d \}=\frac{10^2}{11}(5t_1+5t_2+5t_3+5t_4+5t_5-13).$$ 
Since
$K_YD_p=\frac{20}{11},\,\,\,\frac{6}{11},\,\,\,\frac{5}{11},\,\,\,\frac{32}{11},\,\,\,\frac{81}{11}$
for $p$ of each type,
$$\sum K_{Y}D_p=\frac{1}{11}(100t_1+30t_2+25t_3+160t_4+405t_5).$$ 
Thus we finally have
$$510t_1+690t_2+750t_3+450t_4+150t_5=1278,$$
which has no integer solution, as 1278 is not divisible by 10.

\medskip
Claim:  $n\neq 13$.\\
We omit the proof, which is similar to the previous case.

\medskip
Claim:  $n\neq 7$.\\
Suppose that $n=7$. Then $u=g^{7}$. Let $s=g^{5}.$ By Theorem
\ref{faith}, $\zeta_7$ must appear in $[g^*]$. Suppose that
$[g^*]=[1,\,\zeta_7:6,\,1.7,\,(\zeta_5:4).2].$
Then $[g^{12*}]=[1,\,\zeta_7:6,\,1.7,\,1.8],$ hence the order 7
automorphism $s=g^{5}$ is non-symplectic by Lemma \ref{Lefschetz}
and the quotient $X/\langle g^5\rangle$ is rational. Then by
Proposition \ref{diminv}, $\rank~{\rm NS}(X)^s=\dim H^2_{\rm
et}(X,{\bbQ}_l)^s$. We infer that $[1.7]$ in $[g^*]$ are
supported on ${\rm NS}(X)$, hence $\rank~{\rm NS}(X)^u\ge 8$, contradicting Lemma \ref{1ptK3}. This proves that
$$[g^*]=[1,\,(\zeta_7:6).2,\,1,\,(\zeta_5:4).2].$$  Then $$[g^{12*}]=[g^*],\quad \Tr
(g^{12*}|H^*(X))=0.$$  By Deligne-Lusztig applied to $s=g^5$ and $u=g^7$, $$\Tr(u^*|H^*({\rm Fix}(g^5)))=0.$$ Note that the order 5
action of $u=g^7$ on ${\rm Fix}(g^5)$ has exactly one fixed point.
Since ${\rm Fix}(g^5)$
consists of isolated points, smooth rational curves,  and possibly a curve
$C_r$ of genus $r\ge 2$, the vanishing of the trace implies that the action of $u$ on ${\rm Fix}(g^5)$ is free, a contradiction.

\medskip
Claim:  $n\neq 12$.\\
Suppose that $n=12$. Then $u=g^{12}$ and
$${\rm Fix}(g)={\rm Fix}(g^2)={\rm Fix}(g^3)={\rm Fix}(g^4)={\rm Fix}(g^{6})= {\rm Fix}(g^{12})=\{{\rm one\,\,point}\}.$$
We claim that $g^{30}$ is not symplectic. Suppose that it is. Then by
Proposition \ref{sym} ${\rm Fix}(g^{30})$ consists of 8 points,
on which $g$ acts. This permutation has order dividing 30, so in its cycle decomposition each cycle has length 1, 2, 3, 5, 6, hence  $g^6$
must fix at least 3 points. This proves the claim. Write
$$[g^{*}]=[1,\,(\zeta_{12}:4).a,\,\,(\zeta_4:2).b,\,\,\eta_1,\ldots,\eta_{13-4a-2b},\,\tau_1,\ldots, \tau_{8}]$$ where $\eta_1,\ldots,\eta_{13-4a-2b}$ is a combination of $\pm\zeta_{3}:2$ and $\pm 1$, and  $\tau_1,\ldots, \tau_{8}$ a combination of $\pm\zeta_{5}:4$, $\pm\zeta_{15}:8$, $\zeta_{20}:8$. Since $X'=X/\langle g^{30}\rangle$ is rational, $\rank~{\rm NS}(X)^{g^{30}}=\dim H^2_{\rm
et}(X,{\bbQ}_l)^{g^{30}}$ by
Proposition \ref{diminv}. It follows that $\eta_1,\ldots,\eta_{13-4a-2b}$ are
supported on ${\rm NS}(X)$. Then by Lemma \ref{1ptK3}, $$13-4a-2b\le 3, \,\,i.e.\,\,\,10\le 4a+2b\le 12.$$ We claim that $b\le 1$. Suppose $b\ge 2$. Then  $g^{20}$ is
symplectic. Otherwise, $X/\langle g^{20}\rangle$ is rational and by
Proposition \ref{diminv} $(\zeta_4:2).b$ are
supported on ${\rm NS}(X)$, then ${\rm NS}(X)^{g^{12}}$ would have rank $\ge 1+2b\ge 5$, contradicting Lemma \ref{1ptK3}. Since $g^{20}$ is
symplectic, $\zeta_3:2$ appears 6 times in $[g^{20*}]$. This is possible only if 
$a=1$ and  $[\tau_1,\ldots, \tau_{8}]=[\pm\zeta_{15}:8]$ (hence $b=3$, $4$, and $\eta_i=\pm 1$), i.e. $[g^*]=[1,\,\zeta_{12}:4,\,(\zeta_4:2).3,\,\pm 1,\,\pm 1,\,\pm 1,\,\pm\zeta_{15}:8]$ or $[1,\,\zeta_{12}:4,\,(\zeta_4:2).4,\,\pm 1,\,\pm\zeta_{15}:8]$. Then  $[g^{10*}]=[1,\,(\zeta_6:2).2,\,-1.6,\,1.3,\,(\zeta_3:2).4]$ or $[g^{10*}]=[1,\,(\zeta_6:2).2,\,-1.8,\,1,\,(\zeta_3:2).4]$, hence $e(g^{10})=-2$ or $-6$. Since $e(g^{10})<0$, ${\rm Fix}(g^{10})$ contains a curve of genus $>1$, then so does ${\rm Fix}(g^{20})$, which is a 6 point set.
This proves that $b\le 1$, hence $(a,b)=(3,0)$ or $(2,1)$,
yielding the following 6 cases for $[g^*]$:

\medskip
(1) $[1,\,(\zeta_{12}:4).3,\,\,\pm 1,\,\,\tau_1,\ldots, \tau_{8}]$,

(2) $[1,\,(\zeta_{12}:4).3,\,\,\pm 1,\,\,\zeta_{20}:8]$,

(3) $[1,\,(\zeta_{12}:4).2,\,\,\zeta_4:2,\,\,\pm 1,\,\,\pm 1,\,\,\pm 1,\,\,\tau_1,\ldots, \tau_{8}]$, 

(4) $[1,\,(\zeta_{12}:4).2,\,\,\zeta_4:2,\,\,\pm 1,\,\,\pm 1,\,\,\pm 1,\,\,\zeta_{20}:8]$, 

(5) $[1,\,(\zeta_{12}:4).2,\,\,\zeta_4:2,\,\,\pm\zeta_3:2,\,\,\pm 1,\,\,\tau_1,\ldots, \tau_{8}]$, 

(6) $[1,\,(\zeta_{12}:4).2,\,\,\zeta_4:2,\,\,\pm\zeta_3:2,\,\,\pm 1,\,\,\zeta_{20}:8]$, 

where $\tau_1,\ldots, \tau_{8}$ are combination of $\pm\zeta_{5}:4$, $\pm\zeta_{15}:8$.

\medskip\noindent
Case (1). In this case, $[g^{30*}]=[1,\,-1.12,\,1,\,1.8]$, 
$e(g^{30})=0$, $e(g^{10})=18$ or $6$. Since ${\rm Fix}(g^{5i})$ contains no elliptic curve, ${\rm Fix}(g^{30})$ consists of a
curve $C_{d+1}$ of genus $d+1$ and $d$ smooth rational curves,
where $1\le d\le 9$. Note that $\dim H^2_{\rm
et}(X,{\bbQ}_l)^{g^{6}}=2$. Thus the action of $g^6$ on the $d$ smooth rational curves has at most 1 orbit. This is possible only if $d=1$ of $5$. 
If $d=1$, then $g^6$ acts freely on $C_2$, but no genus 2 curve admits an order 5 free action. If $d=5$, then $g^{10}$ preserves
each of the $5$ smooth rational curves and, since $e(g^{10})=18$
or $6$, fixes 8 points on  $C_{6}$. Then $g^2$ acts on these 8 points, hence fixes at least 3 of them.

\medskip\noindent
Case (2). In this case, $[g^{30*}]=[1,\,-1.12,\,1,\,-1.8]$, 
$e(g^{30})=-16$. Thus ${\rm Fix}(g^{30})$ consists
of either a curve of genus 9 or a curve of genus $10$ and a smooth rational curve. Note that $e(g^{10})=2$, so the tame involution $g^{10}$
has 2 fixed points on ${\rm Fix}(g^{30})$. Then $g^2$ acts on these 2 points, hence fixes both.

\medskip\noindent
Case (3), (4). In these cases, $e(g^{20})=12$ or 0. Thus $g^{20}$ is not symplectic, 
$X/\langle g^{20}\rangle$ is rational, and by
Proposition \ref{diminv} $[\zeta_4:2,\,\pm 1\,\pm 1\,\pm 1]$ in $[g^*]$ are
supported on ${\rm NS}(X)$. Then ${\rm NS}(X)^{g^{12}}$ has rank $\ge 1+5$, contradicting Lemma \ref{1ptK3}.

\medskip\noindent
Case (5). In this case, $[g^{30*}]=[1,\,-1.8,\,-1.2,\,1.3,\,1.8]$, 
$e(g^{30})=4$. Thus ${\rm Fix}(g^{30})$ consists of either 2 smooth rational curves or a curve $C_{d-1}$ of genus $d-1$ and $d$ smooth rational
curves, $3\le d\le 11$. In the first case, $g^2$ acts on ${\rm Fix}(g^{30})$, hence preserves each of the 2 rational curves, then fixes at least 2 point on them. Assume the second case. Note that $\dim H^2_{\rm
et}(X,{\bbQ}_l)^{g^{2}}=2$. Thus the action of $g^2$ on the $d$ smooth rational curves has at most 1 orbit. The length of an orbit divides 15. Thus $d=3$ of $5$. 
If $d=3$, then $g^6$ preserves each of the 3 rational curves, hence fixes at least 3 points on them. If $d=5$, then $g^{10}$ preserves
each of the $5$ smooth rational curves and, since $e(g^{10})=13$
or $1$, fixes 3 points on  $C_{4}$. Then $g^2$ acts on these 3 points, hence fixes all of them.

\medskip\noindent
Case (6). In this case, $[g^{30*}]=[1,\,-1.8,\,-1.2,\,1.3,\,-1.8]$, 
$e(g^{30})=-12$, $e(g^{10})=-3$. Thus ${\rm Fix}(g^{30})$ consists of a curve $C_{d+7}$ of genus $d+7$ and $d$ smooth rational
curves, $0\le d\le 3$. 
Since $e(g^{10})<0$, ${\rm Fix}(g^{10})$ must contain $C_{d+7}$, hence  $e(g^{10})\le  e(g^{30})$, i.e. $-3\le -12$, absurd.
\end{proof}

\begin{example}\label{exam5-2} In char $p = 5$, there are K3 surfaces with an automorphism of order 30 or 20.

\medskip
(1) $X_{30}: y^2=x^3+(t^{5}-t)^2$, $g_{30}(t,x,y)=(t+1,\zeta_3 x,
-y)$;

\medskip
(2) $X_{20}: y^2=x^3+(t^{5}-t)x$, $g_{20}(t,x,y)=(t+1, -x, \zeta_4
y)$.

\medskip\noindent
The surface $X_{30}$ has 6
 $IV$-fibres at $t=\infty$, $t^5-t=0$; $X_{20}$ has a
 $III^*$-fibre at $t=\infty$ and 5 $III$-fibres at $t^5-t=0$
 (\cite{DK1}, 5.8).
\end{example}

\textit{Acknowledgement}

\medskip
I am grateful to Igor Dolgachev, Toshiyuki Katsura, Shigeyuki
Kond\=o, Keiji Oguiso, Matthias Sch\"utt and Weizhe Zheng for helpful comments. A
part of this paper was written during my stay at Roma Tre in
February 2012, I would like to thank Edoardo Sernesi for his
hospitality.

%%%%%%%%%%%%%%%%%%%%%%%%%%%%%%%%%%%%%%%%%%%%%%%%%%%%%%%%%%%%%%%%%%%%%%%%
%\bibliographystyle{amsplain}


\begin{thebibliography}{99}




%\bibitem{Artin}  M. Artin, \textit{Supersingular K3 surfaces}, Ann. Ec. Norm. Sup., 4-e Serie, {\bf 7} (1974), 543--567

\bibitem{Artin}  M. Artin, \textit{Coverings of the rational double points in characteistic $p$}, in Complex Analysis and Algebraic Geometry, Iwanami Shoten, Tokyo, 1977,
11--22


\bibitem{CD} F. Cossec,  I. Dolgachev,  \textit {Enriques surfaces I}, Birkh\"auser 1989

\bibitem{CoxZ}  D. Cox, S. Zucker, \emph{Intersection numbers of sections of elliptic surfaces}, Invent. Math. {\bf 53} (1979), 1--44

\bibitem{De} P. Deligne, \textit{ Rel\`evement des surfaces K3 en caract\'eristique nulle}, in Surface alg\'ebrique, edt J. Giraud, L, Illusie and M. raynaud, Lecture Notes in Math. {\bf 868}, Springer (1981), 58-79.

\bibitem{DL} P. Deligne, G. Lusztig, \textit{ Representations of reductive groups over finite fields}, Ann. of Math. (2) {\bf 103} (1976), no. 1, 103--161.

\bibitem{DK1} I. Dolgachev, J. Keum, \textit{ Wild
$p$-cyclic actions on K3 surfaces}, J. Algebraic Geometry, {\bf
10} (2001), 101--131

\bibitem{DK2} I. Dolgachev, J. Keum, \textit{Finite groups of symplectic
automorphisms of K3 surfaces in positive characteristic}, Ann. of
Math. {\bf 169} (2009), 269--313

\bibitem{DK3} I. Dolgachev, J. Keum, \textit{K3 surfaces with a symplectic
automorphism of order $11$}, J. Eur. Math. Soc. {\bf 11} (2009),
799--818

\bibitem{HK1} D. Hwang and J. Keum, \textit{The maximum number of singular points on rational homology projective planes}, J. Algebraic
Geom. {\bf 20} (2011), 495--523.

\bibitem{Ill}  L. Illusie, \textit{Report on crystalline
cohomology}, in ``Algebraic Geometry, Arcata 1974", Proc. Symp.
Pure math. vol. 29 , AMS, pp. 459--478

\bibitem{Ito} H. Ito, \textit{On automorphisms of supersingular K3 surfaces}, Osaka J.
Math.  {\bf 34} (1997), 717--724

\bibitem{ItoL} H. Ito, Ch. Liedtke, \textit{Elliptic K3 surfaces with $p^n$-torsion secions},  arXiv:1003.0144v3 [math.AG]

\bibitem{K} J. Keum, \textit{Automorphisms of Jacobian Kummer surfaces}, Compositio Math. {\bf 107} (1997), 269--288

\bibitem{K2} J. Keum, \textit{K3 surfaces with an order $60$
automorphism and a characterization of supersingular K3 surfaces with Artin invariant 1}, arXiv:1204.1711 [math.AG]

\bibitem{Ko} S. Kond\=o, \textit{Automorphisms of algebraic K3 surfaces which act trivially on Picard groups}, J. Math. Soc. Japan
{\bf 44} (1992), 75--98

\bibitem{Ko2} S. Kond\=o, \textit{The maximum order of finite groups of automorphisms of K3
surfaces}, Amer. J. Math. {\bf 121} (1999), no. 6, 1245--1252.
\bibitem{Ko3} S. Kond\={o}, \textit{Maximal subgroups of the Mathieu group $M_{23}$ and symplectic automorphisms of supersingular K3 surfaces}, IMRN
{\bf 2006}, Art ID 71517, 9pp.

\bibitem{MO} N. Machida, K. Oguiso, \textit{On K3 surfaces admitting finite non-symplectic group actions}, J. Math. Sci. Univ. Tokyo
{\bf 5} (1998), 273--297

\bibitem{Muk}  S. Mukai, \textit {Finite groups of automorphisms of
$K3$ surfaces and the Mathieu group}, Invent. Math. {\bf 94}
(1988), 183--221.

\bibitem{Nik} V. V. Nikulin, \textit{Finite groups of automorphisms of K\"ahlerian surfaces of type K3}, Uspehi Mat. Nauk {\bf 31} (1976), no. 2; Trans. Moscow Math. Soc., {\bf 38} (1980), 71--135

\bibitem{Ny}  N. Nygaard, \textit {Higher DeRam-Witt complexes on supersingular K3 surfaces}, Compositio Math.  {\bf 42}
(1980/81), 245--271

\bibitem{Og1} K. Oguiso, \textit{A remark on global indices of $\bbQ$-Calabi-Yau 3-folds}, Math. Proc. Camb. Phil. Soc.
{\bf 114} (1993), 427--429

\bibitem{Ogus} A. Ogus,
\textit{Supersingular K3 crystals}, in ``Journ\'ees de G\'eometrie
Alg\'ebrique de Rennes'', Asterisque, Vol. {\bf 64} (1979), pp. 3--86


%\bibitem{PS} I. I. Piateckii-Shapiro and I. R. Shafarevich, \textit{A Torelli theorem for algebraic surfaces of type K3}, Izv. Akad. Nauk SSSR Ser. Mat. {\bf 35} (1971), 530--572; Math USSR Izvestija Vol. {\bf 5} (1971), 547--588

\bibitem{RS} A.\ Rudakov, I.\ Shafarevich, {\it Surfaces of type K3
over fields of finite characteristic}, Current problems in
mathematics, vol. 18, pp. 115-207, Akad. Nauk SSSR, Vsesoyuz.
Inst. Nauchn. i Tekhn. Inform., Moscow, 1981 (reprinted in I. R.
Shafarevich, ''Collected Mathematical Papers'', Springer-Verlag,
1989, pp.657--714).

\bibitem{Serre} J.-P. Serre, \textit{Le groupe de Cremona et ses sous-groupe finis}, S\'eminaire Bourbaki, Vol. 2008/2009, Expos\'es 997-1011,
Ast\'erisque No. 332 (2010), Exp. No. 1000, 75-100

\bibitem{Shioda2}  T. Shioda, \textit{An explicit algorithm for computing the Picard number of certain algebraic surfaces},
Amer. J. Math. {\bf 108} (1986), 415--432

\bibitem{Shioda3}  T. Shioda, \emph{On the Mordell-Weil lattices}, Commemtarii Mathematici Univ. Sancti Pauli {\bf 39} (1990), 211--240
\bibitem{Silverman} J. Silverman,  \textit{The arithmetic of elliptic curves}, Grad. Texts in Math. 106, Springer-Verlag
1986

\bibitem{Ueno} K. Ueno, \textit{Classification theory of algebraic varieties and compact complex spaces},
Springer Lect. Notes in Math. {\bf 439} (1975)

\bibitem{Ueno86} K. Ueno, \textit{A remark on automorphisms of Kummer surfaces in characteristic
$p$}, J. Math. Kyoto Univ. {\bf 26} (1986), 483--491

\bibitem{Xiao} G. Xiao, \textit{Non-symplectic involutions of a K3 surface}, unpublished

\end{thebibliography}
\end{document}